\documentclass[11pt,reqno]{amsart}
\usepackage{enumerate}
\usepackage{amsmath,amssymb,amsfonts,amsthm, amscd,indentfirst}
\usepackage{amsmath,latexsym,amssymb,amsmath,amscd,amsthm,amsxtra}
\usepackage{hyperref}\usepackage{url}
\hypersetup{hypertex=true,colorlinks=true,linkcolor=blue,anchorcolor=blue,citecolor=blue}
\usepackage{color}
\usepackage{pdfpages}
\usepackage{graphics}
\usepackage{enumitem}
\usepackage{tikz}
\usepackage{graphicx}
\usepackage[font=small]{caption}
\usepackage{subcaption}
\usepackage{epsfig,here}
\usepackage{cite}
\numberwithin{equation}{section}
\allowdisplaybreaks[3]

\renewcommand{\thefigure}{\thesection.\arabic{figure}}

\newtheorem{Theorem}{Theorem}[section]
\newtheorem{Lemma}{Lemma}[section]
\newtheorem{Proposition}{Proposition}[section]

\newtheorem{Remark}{\textbf{Remark}}[section]

\def\<{\langle}
\def\>{\rangle}

\allowdisplaybreaks
\begin{document}
\title[Characterizations of umbilical hypersurfaces by partially overdetermined problems]{Characterizations of umbilical hypersurfaces by partially overdetermined problems in space forms}

\author{Yangsen Xie}
\address{Yangsen Xie, School of Mathematical Sciences\\Xiamen University\\361005, Xiamen, P. R. China}
\email{19020211153416@stu.xmu.edu.cn}

\begin{abstract}
    In this paper, we characterize the rigidity of umbilical hypersurfaces by a Serrin-type partially overdetermined problem in space forms, which generalizes the similar results in Euclidean half-space \cite{jia2023rigidity} and Euclidean half-ball \cite{jia2023characterization}. Guo-Xia first obtained these rigidity results \cite{guo2019partially,guo2022partially} when the Robin boundary condition on the support hypersurface is homogeneous, at this time the target umbilical hypersurface has orthogonal contact angle with the support. However, in this paper we can obtain any contact angle $\theta\in (0,\pi)$ by changing the Robin boundary condition to be inhomogeneous.
\end{abstract}
	
\date{\today}
	
\maketitle
		
\section{Introduction}
	In \cite{serrin1971symmetry}, Serrin initiated the study of overdetermined boundary value problem (BVP) in a bounded domain $\Omega\subset\mathbb{R}^{n+1}$:
\begin{align}\label{Serrin}
\left\{
\begin{aligned}
  \overline{\triangle} u=1 \quad&\rm{in} \;\Omega,\\
  u=0 \quad&\rm{on}\;\partial\Omega,\\
  \partial_\nu u=c \quad&\rm{on}\; \partial\Omega,
\end{aligned}
\right.
\end{align}
where $c\in\mathbb{R}$ is a constant and $\nu$ is the unit outward normal to $\partial\Omega$. He proved that \eqref{Serrin} admits a solution if and only if $\partial\Omega$ is a round sphere and the solution $u$ is radically symmetric. Serrin's proof is based on the moving plane method, which was invented by Alexandrov when he proved the well-known Alexandrov soap bubble theorem \cite{aleksandrov1962uniqueness}: any embedded closed hypersurface of constant mean curvature (CMC) in $\mathbb{R}^{n+1}$ must be a round sphere. Soon after Serrin's paper, Weinberger \cite{weinberger1971remark} gave a new proof by means of more elementary arguments. Due to the works of Serrin and Weinberger, a number of overdetermined problems of general elliptic equations in Euclidean were studied, one may refer to \cite{birindelli2013overdetermined, brandolini2008serrin, buttazzo2011overdetermined, cianchi2009overdetermined, farina2008remarks, farina2010flattening, fragala2006overdetermined, garofalo1989symmetry, lu2012overdetermined, wang2011characterization}. Serrin's symmetry result was generalized to space forms in \cite{kumaresan1998serrin,molzon1991symmetry} by using the moving plane method.

    In \cite{guo2019partially,guo2022partially}, Guo-Xia studied a Serrin-type partially overdetermined BVP in space forms $\mathbb{M}^{n+1}(K)$ ($K=0,-1,1$). Precisely, let $S_{K,\kappa}$ be an umbilical hypersurface with principal curvature $\kappa>0$ in $\mathbb{M}^{n+1}(K)$, $\Sigma$ be a hypersurface supported on $S_{K,\kappa}$, $\Omega$ be the domain enclosed by $\Sigma$ and $S_{K,\kappa}$. Then, they considered the following partially overdetermined BVP in $\Omega$,
\begin{align}
\left\{
\begin{aligned}
\overline{\triangle}u+(n+1)Ku=1\quad&\text{in}\,\,\Omega,\\
u=0\quad&\text{on}\,\,\overline{\Sigma},\\
\partial_\nu u=c\quad&\text{on}\,\,\overline{\Sigma},\\
\partial_{\bar{N}}u=\kappa u\quad&\text{on}\,\,\partial\Omega\backslash\overline{\Sigma},
\end{aligned}
\right.\label{Guo}
\end{align}
where $c\in\mathbb{R}$ is a constant, $\nu$ and $\bar{N}$ are the outward unit narmal to $\Sigma$ and $S_{K,\kappa}$ respectively. By using purely integral method, Guo-Xia obtained the following theorem:
\begin{Theorem}
Assume the partially overdetermined BVP \eqref{Guo} admits a weak solution $u\in W_0^{1,2}(\Omega,\Sigma)$, i.e.
$$\int_\Omega\left[\bar{g}(\overline{\nabla}u,\overline{\nabla}v)+v-(n+1)Kuv\right]dx=\kappa\int_{\partial\Omega\backslash\overline{\Sigma}}uvdA,\,\,\forall\,v\in W_0^{1,2}(\Omega,\Sigma),$$
together with an additional boundary condition $\partial_\nu u=c$ on $\Sigma$. Assume further that $u\in W^{1,\infty}(\Omega)\cap W^{2,2}(\Omega)$. Then $c>0$ and $\Sigma$ must be part of an umbilical hypersurface with principal curvature $\frac{1}{(n+1)c}$ which intersects $S_{K,\kappa}$ orthogonally.
\label{guotheorem}
\end{Theorem}

\begin{Remark}
\begin{enumerate}
  \item[\rm (i)] Actually, $\Omega$ should lie in a more precise domain: $B_{K,\kappa}^{\rm{int}}$ or $B_{K,\kappa}^{\rm{int},+}$. When $S_{K,\kappa}$ is a geodesic sphere, we use $B_{K,\kappa}^{\rm{int},+}$, for other types of umbilical hypersurface we use $B_{K,\kappa}^{\rm{int}}$ (See Section \ref{section2} below for the details). For notation simplicity and unification, we emphasize here that: in the following, we use $\Omega\subset B_{K,\kappa}^{\rm{int}}$ to indicate that $\Omega\subset B_{K,\kappa}^{\rm{int},+}$ in the case that $S_{K,\kappa}$ is a geodesic sphere.
  \item[\rm (ii)] The case $\kappa=0$ in $\mathbb{M}^{n+1}(K)$ was also solved in \cite{ciraolo2020serrin} (as a special case of a flat cone).
\end{enumerate}

\end{Remark}

    After the works of Guo-Xia, Jia-Lu-Xia-Zhang replaced the Robin condition in \eqref{Guo} with $\partial_{\bar{N}}u=\kappa u+\widetilde{c}$ and studied the new problem in \cite{jia2023characterization,jia2023rigidity}, where $\widetilde{c}$ is also a constant. Precisely, they considered the following new partially overdetermined problem in a bounded domain $\Omega$ in half-space $\mathbb{R}_+^{n+1}$ or half-ball $\mathbb{B}_+^{n+1}$ ($K=0$):
\begin{align}
\left\{
\begin{aligned}
\overline{\triangle}u=1\quad&\text{in}\,\,\Omega,\\
u=0\quad&\text{on}\,\,\overline{\Sigma},\\
\partial_\nu u=c\quad&\text{on}\,\,\overline{\Sigma},\\
\partial_{\bar{N}}u=\kappa u+\widetilde{c}\quad&\text{on}\,\,T,
\end{aligned}
\right.\label{Jia}
\end{align}
where $\Sigma=\partial\Omega\cap\mathbb{R}_+^{n+1}$ or $\Sigma=\partial\Omega\cap\mathbb{B}_+^{n+1}$, $T=\partial\Omega\backslash\overline{\Sigma}$, $\kappa$ is the principal curvature of the boundary hypersurface of $\mathbb{R}_+^{n+1}$ ($\kappa=0$) or $\mathbb{B}^{n+1}$ ($\kappa=1$). By establishing new integral identities, they proved the following Serrin-type rigidity result:
\begin{Theorem}
If the partially overdetermined problem \eqref{Jia} admits a weak solution $u\in W^{1,\infty}(\Omega)\cap W^{2,2}(\Omega)$ such that $u\leq 0$, then $c>0$ and $\Sigma$ must be a spherical cap with radius $(n+1)c$ which intersects $\partial\mathbb{R}_+^{n+1}$ or $\partial\mathbb{B}^{n+1}$ at a constant contact angle $\theta$, characterized by $\cos\theta=-\frac{\widetilde{c}}{c}$.
\label{Jiatheorem}
\end{Theorem}

    The aim of this paper is to generalize Theorem \ref{Jiatheorem} into the setting of domains with partial umbilical boundary in space forms.

    Let $(\mathbb{M}^{n+1}(K),\bar{g})$ be a complete simply-connected Riemannian manifold with constant sectional cunvature $K$. We only need to consider
$K=0,1,-1$ by scaling. $K=0$ corresponds to the Euclidean space $(\mathbb{R}^{n+1},\delta)$, $K=1$ the unit sphere $\mathbb{S}^{n+1}$ with round metric, and $K=-1$ the hyperbolic space $\mathbb{H}^{n+1}$.

    We state some facts about umbilical hypersurfaces in $\mathbb{M}^{n+1}(K)$. The principal curvature of an umbilical hypersurface is a constant
$\kappa\in\mathbb{R}$. We may assume $\kappa\geq 0$ by choosing a normal vector field $\bar{N}$ of the hypersurface. As is well-known, all umbilical hypersurfaces in $\mathbb{R}^{n+1}$ and $\mathbb{S}^{n+1}$ are geodesic spheres ($\kappa>0$) and totally geodesic hyperplanes ($\kappa=0$). The cases in $\mathbb{H}^{n+1}$ are a bit more complicated, all of which are listed: geodesic spheres ($\kappa>1$), totally geodesic hyperplanes ($\kappa=0$), horospheres ($\kappa=1$) and equidistant hypersurfaces ($0<\kappa<1$). For visualizations of these umbilical hypersurfaces, see Section \ref{section2}.

    We denote an umbilical hypersurface in $\mathbb{M}^{n+1}(K)$ with principal curvature $\kappa$ as $S_{K,\kappa}$, then it will divides $\mathbb{M}^{n+1}(K)$
into two regions. We use $B_{K,\kappa}^{\rm int}$ to denote the region whose outward normal is $\bar{N}$ that has been selected before, and denote the other one by $B_{K,\kappa}^{\rm ext}$.

    Let $\Omega\subset B_{K,\kappa}^{\rm int}$ be a bounded domain with boundary $\partial\Omega=\overline{\Sigma}\cup T$, where $\Sigma\subset B_{K,\kappa}^{\rm int}$ is a smooth open hypersurface and $T=\partial\Omega\backslash\overline{\Sigma}\subset S_{K,\kappa}$ meets $\overline{\Sigma}$ at a common $(n-1)$-dimensional submanifold $\Gamma\subset S_{K,\kappa}$. We consider the following mixed BVP:
\begin{align}\label{BVP}
\left\{
\begin{aligned}
  \overline{\triangle}u+(n+1)Ku=1 \quad&\text{in} \,\,\Omega,\\
  u=0 \quad&\text{on}\,\,\overline{\Sigma},\\
  \partial_{\bar{N}}u=\kappa u+\widetilde{c}\quad&\text{on}\,\, T,
\end{aligned}
\right.
\end{align}
where $\widetilde{c}\in\mathbb{R}$ is a constant.

\begin{Remark}
    \begin{enumerate}
      \item[\rm (i)] There exists a unique weak solution $u\in W_0^{1,2}(\Omega,\Sigma)$ to \eqref{BVP} for every $\kappa\geq 0$. The case $\kappa>0$ is proved in \cite[Propostion 3.3]{guo2022partially}, and we add the proof of the case $\kappa=0$ in \hyperlink{AppendixB}{Appendix B}.
      \item[\rm (ii)] By the classical regularity theory for elliptic equations, we know $u\in C^\infty(\overline{\Omega}\backslash\Gamma)$, and $u\in C^\alpha(\overline{\Omega})$ is proved by Liebermann (See \cite[Theorem 2]{lieberman1986mixed}).
    \end{enumerate}
\end{Remark}

In this paper, we study the following partially overdetermined BVP in $\Omega\subset B_{K,\kappa}^{\text{int}}$:
\begin{align}\label{Xie}
\left\{
\begin{aligned}
  \overline{\triangle}u+(n+1)Ku=1 \quad&\text{in}\,\,\Omega,\\
  u=0 \quad&\text{on}\;\overline{\Sigma},\\
  \partial_\nu u=c \quad&\text{on}\,\,\overline{\Sigma},\\
  \partial_{\bar{N}}u=\kappa u+\widetilde{c} \quad&\text{on}\,\,T,
\end{aligned}
\right.
\end{align}
where $c$ is also a constant. We will prove the following Serrin-type theorem:

\begin{Theorem}
    Assume the partially overdetermined BVP \eqref{Xie} admits a weak solution $u\in W^{1,\infty}(\Omega)\cap W^{2,2}(\Omega)$ such that $u\leq 0$, then $c>0$ and $\Sigma$ must be part of an umbilical hypersurface with principal curvature $\frac{1}{(n+1)c}$ which intersects $S_{K,\kappa}$ at a constant contact angle $\theta$, characterized by $\cos\theta=-\frac{\widetilde{c}}{c}$.
\label{Theorem}
\end{Theorem}
\begin{Remark}
    \begin{enumerate}
      \item[\rm (i)] The case $\widetilde{c}=0$ is solved by Guo-Xia in \cite{guo2019partially,guo2022partially}, which corresponds to the orthogonal angle $\theta=\pi/2$.
      \item[\rm (ii)] Under the condition $u\leq 0$, we can further know that $u<0$ in $\Omega$ otherwise it will contradict the equation $\overline{\triangle}u+(n+1)Ku=1$. If $\widetilde{c}\leq 0$, the condition $u\leq 0$ is automatically satisfied (see \cite[Proposition 3.4]{guo2022partially}).
      \item[\rm (iii)] The case $\Omega\subset B_{K,\kappa}^{\rm ext}$ with unit outward normal $-\bar{N}$ along $T$ can also be discussed and have similar conclusion. However, the Robin boundary condition in \eqref{Xie} should become $\partial_{-\bar{N}}u=-\kappa u+\widetilde{c}$ on $T$ ($-\kappa\leq 0$).
    \end{enumerate}
\end{Remark}

\textbf{The idea of the proof of Theorem \ref{Theorem}:} Similar to the cases of half-space and half-ball in Euclidean \cite{jia2023characterization,jia2023rigidity}, with the classical $P$-function $P_u:=\frac{1}{2}\bar{g}(\overline{\nabla}u,\overline{\nabla}u)-\frac{1}{n+1}u+\frac{K}{2}u^2$ and an auxiliary function $\varphi$, we establish a integral identity for each umbilical hypersurface case:
\begin{align}\label{identity}
\begin{aligned}
&\int_\Omega-Vu\,\overline{\triangle}P_udx\\
=&\frac{1}{2}\int_\Sigma\left(\bar{g}(\overline{\nabla}u,\overline{\nabla}u)-a\right)\left[V(u-\varphi)_\nu-V_\nu(u-\varphi)\right]dA.
\end{aligned}
\end{align}
Here $V$ is a positive function which is given by a multiplier of the divergence of a conformal Killing vector field $X$ and $a$ is an arbitrary constant. See Section \ref{section2} for the detail expressions of $X$ and $V$. The auxiliary function $\varphi$ is a function in $C^{\infty}(\mathbb{M}^{n+1}(K))$ that satisfies:
\begin{align}\label{resolvent}
\left\{
\begin{aligned}
  \overline{\nabla}^2\varphi=\left(\frac{1}{n+1}-K\varphi\right)\bar{g}\quad&\text{on}\,\,\mathbb{M}^{n+1}(K)\\
  \partial_{\bar{N}}\varphi=\kappa \varphi+\widetilde{c} \quad&\text{on}\,\,S_{K,\kappa}.\\
\end{aligned}
\right.
\end{align}
The advantage of introducing the auxiliary function $\varphi$ is that some integrals on the boundary $T$ will vanish in calculations. The hypersurface $\Sigma$ being umbilical follows immediately if we take $a=c^2=\bar{g}(\overline{\nabla}u,\overline{\nabla}u)|_\Sigma$ in \eqref{identity} and use the subharmonicity of $P_u$. The contact angle characterization can be obtained by using some integral formulas and a proposition about the mean curvature.

\textbf{Organization of the paper:} In Section \ref{section2}, we collect the conformal Killing vector fields $X$ and the Killing vector fields $Y$ we shall use in each umbilical hypersurface case and their properties. In Section \ref{section3}, we find an auxiliary function $\varphi$ to the mixed BVP \eqref{BVP} for each case and give some integral formulas and a proposition about the mean curvature. In Section \ref{section4}, we prove the crucial integral identities \eqref{identity} and then Theorem \ref{Theorem}.

\section{Preliminaries}\label{section2}
    Assume $B$ be a domain in $\mathbb{M}^{n+1}(K)$. Let $\Omega$ be a bounded domain in $B$ with piece-wise smooth boundary $\partial\Omega=\overline{\Sigma}\cup T$, where $\Sigma\subset B$ is a smooth open hypersurface and $T=\partial\Omega\backslash\overline{\Sigma}\subset\partial B$ meets $\overline{\Sigma}$ at a common $(n-1)$-dimensional submanifold $\Gamma\subset\partial B$.

    We denote by $\overline{\nabla}$, $\overline{\triangle}$, $\overline{\nabla}^2$, $\overline{\rm{div}}$ the Riemannian connection, the Laplacian, the Hessian and the divergence on $(\mathbb{M}^{n+1}(K),\bar{g})$ respectively. Let $\nabla^\Sigma, \rm{div}^\Sigma$ be the Riemannian connection and the divergence on the hypersurface $\Sigma$, $h$ the second fundamental form of $\Sigma$ and $H$ the mean curvature of $\Sigma$.

    We denote by $\nu$ and $\bar{N}$ the unit outward normal to $\Sigma$ and $T$ (with respect to $\Omega$) respectively. Let $\mu$ be the unit outward co-normal to $\Gamma\subset\overline{\Sigma}$, $\bar{\nu}$ be the unit outward co-normal to $\Gamma\subset \overline{T}$. Then along $\Gamma$, $\{\nu,\mu\}$ and $\{\bar{\nu},\bar{N}\}$ have the same orientation in the normal bundle of $\Gamma\subset B$ (See Figure \ref{domain}). In particular, if $\Sigma$ meets $\partial B$ at a constant contact angle $\theta$, then along $\Gamma$,
\begin{equation}\label{vectors4}
    \mu=\sin\theta\bar{N}+\cos\theta\bar{\nu},\quad \nu=-\cos\theta\bar{N}+\sin\theta\bar{\nu}.
\end{equation}

\begin{center}
\begin{figure}[h]
\vspace{-10pt}
\begin{tikzpicture}[scale=2.2]
    \draw (0,0) to [out=-45,in=160] (0.7,-0.75) to [out=-20,in=180] (2.3,-1) to [out=0,in=-160] (3,-0.85);
    \draw[cyan,thick,fill=gray!25] (0.7,-0.75) to [out=-20,in=180] (2.3,-1);
    \draw[red,thick,fill=gray!25] (0.7,-0.75) to [out=60,in=110] (2.3,-1);
    \draw[->] (2.3,-1) to (2.3,-1.4);
    \draw[->] (2.3,-1) to (2.7,-1);
    \draw[->] (2.3,-1) to (2.4368,-1.3795);
    \draw[->] (2.3,-1) to (2.6759,-0.8632);
    \node at (1.5,0.15) {$B$};\node[left] at (0,0) {$\partial B$};
    \node[left] at (2.3,-1.4) {$\bar{N}$};
    \node[below] at (2.7,-1) {$\bar{\nu}$};
    \node[right] at (2.4368,-1.3795) {$\mu$};
    \node[above] at (2.6759,-0.8632) {$\nu$};
    \node at (1.5,-0.7) {$\Omega$};\node[red] at (1.65,-0.35) {$\Sigma$};\node[cyan] at (1.3,-1.07) {$T$};
    \draw (2.4,-1) arc [radius=0.1,start angle=0,end angle=-70];\node[above] at (2.45,-1.25) {$\theta$};
\end{tikzpicture}
\caption{}
\label{domain}
\vspace{-15pt}
\end{figure}
\end{center}

    Next, we will give the visualizations of all umbilical hypersurfaces in $\mathbb{M}^{n+1}(K)$ ($K=\pm 1$). Each umbilical hypersurface case is associated with a conformal Killing vector field $X$ and a Killing vector field $Y$. The divergence of $X$ will produce the nonnegative function $V$ in the integral identity \eqref{identity}, and $Y$ is used to the formulation of some integral formulas in Section \ref{section3}.

    There are two models for the hyperbolic space $\mathbb{H}^{n+1}$, the Poincar\'{e} ball model $$\mathbb{B}^{n+1}=\left\{x\in\mathbb{R}^{n+1}\,\big|\,|x|<1\right\},\quad \bar{g}=\frac{4}{(1-|x|^2)^2}
\delta,$$ and the upper half-space model $$\mathbb{R}_+^{n+1}=\left\{x\in\mathbb{R}^{n+1}\,\big|\,x_{n+1}>0\right\}, \quad\bar{g}=\frac{1}{x_{n+1}^2}
\delta.$$ For the spherical space $\mathbb{S}^{n+1}$, we use the stereographic projection model $$\left(\mathbb{R}^{n+1},\,\bar{g}=\frac{4}{(1+|x|^2)^2}\delta\right),$$ which represent $\mathbb{S}^{n+1}\backslash\{S\}$, the unit sphere without the south pole. Let $E_1,\cdots,E_{n+1}$ form the coordinate frame in the Euclidean metric.

    \cite{guo2022partially} have given the information about geodesic spheres, horospheres, equidistant hypersurfaces in $\mathbb{H}^{n+1}$ and geodesic spheres in $\mathbb{S}^{n+1}$. \cite{chen2022some} have given the information about totally geodesic hyperplanes in $\mathbb{H}^{n+1}$. For the convenience of readers, we list them as follows.

  \textbf{Case 1:} If $S_{K,\kappa}$ is a geodesic sphere of radius $R\in(0,\infty)$ in $\mathbb{H}^{n+1}$, then $\kappa=\coth R\in(1,\infty)$. Let $B_{K,\kappa}^{\rm int}$ denote the geodesic ball enclosed by $S_{K,\kappa}$. Using the Poincar\'{e} ball model, up to an hyperbolic isometry, we have
$$B_{K,\kappa}^{\rm int}=\left\{x \in \mathbb{B}^{n+1}\, \Bigg| \, |x|< R_{\mathbb{R}}:=\sqrt{\frac{\cosh R -1}{\cosh R +1}}\right\}.$$ Moreover, we let $B_{K,\kappa}^{\rm{int},+}=\left\{x\in B_{K,\kappa}^{\rm int}\,\Big|\,x_{n+1}>0\right\}$ be the corresponding geodesic half-ball (See Figure \ref{geodesic}).
\begin{center}
\begin{figure}[h]
\begin{tikzpicture}[scale=1.3]
    \draw[fill=gray!25] (1.6,0) arc [radius=1.6,start angle=0, end angle=180];
    \node[below] at (2.57,0) {\footnotesize $1$};
    \draw[dashed,thick] (0,0) circle [radius=2.5];\draw (1.6,0) arc [radius=1.6,start angle=0,end angle=-180];
    \draw (-2.5,0) to (2.5,0);\node[below] at (0.1,0) {$o$};\draw[->] (0,0) to (0,2.6);\node[left] at (0,2.6) {\footnotesize{$x_{n+1}$}};
    \draw[dashed,<-] (-1.6,-0.1) to (-1,-0.1);\draw[dashed,->] (-0.6,-0.1) to (0,-0.1);\node at (-0.8,-0.14) {\tiny{$R_{\mathbb{R}}$}};
    \draw[thick,red] (1.3856,0.8) to [out=-150,in=-60] (-0.8,1.3856);\draw[thick,cyan] (1.3856,0.8) arc [radius=1.6,start angle=30,end angle=120];
    \node[red] at (-0.65,0.9) {\footnotesize$\Sigma$};\node[cyan] at (0.6,1.64) {\footnotesize$T$};\node at (0.3,1) {\footnotesize$\Omega$};
    \node at (0.8,0.3) {\footnotesize$B_{K,\kappa}^{\text{int},+}$};\node at (1.6,1.5) {\footnotesize$B_{K,\kappa}^{\text{ext}}$};
    \node at (0,-1.8) {\footnotesize $S_{K,\kappa}$};
\end{tikzpicture}
\caption{$S_{K,\kappa}$ is a geodesic sphere with $\kappa=\coth R>1$, and the shaded area is $B_{K,\kappa}^{\rm{int},+}$.}
\label{geodesic}
\vspace{-15pt}
\end{figure}
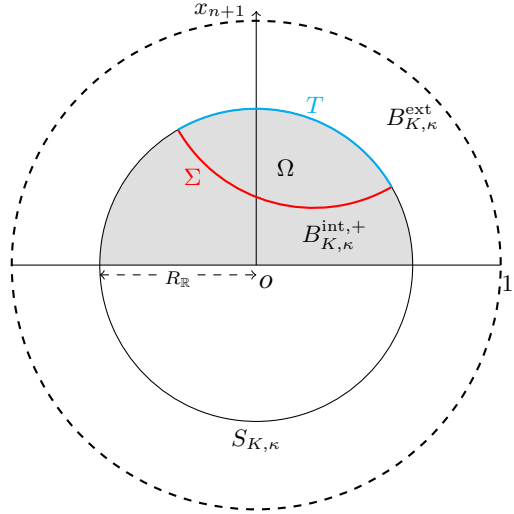
\end{center}

    In this case, set
\begin{equation}
    X_1:=\frac{2}{1-R_{\mathbb{R}}^2}\left[x_{n+1}x-\frac{|x|^2+R_{\mathbb{R}}^2}{2}E_{n+1}\right],~V_1:=\frac{2x_{n+1}}{1-|x|^2},~Y_1:=\frac{1+|x|^2}{2}E_{n+1}-x_{n+1}x.
\label{XVY1}
\end{equation}

    \textbf{Case 2:} If $S_{K,\kappa}$ is an equidistant hypersurface in $\mathbb{H}^{n+1}$, then $\kappa\in(0,1]$. We remark that $\kappa=1$ when $S_{K,\kappa}$ is a horosphere. Using the upper half-space model, we have up to an hyperbolic isometry,
$$S_{K,\kappa}=L_\alpha=\left\{x\in\mathbb{R}_+^{n+1}\,|\,x_1\tan\alpha+x_{n+1}=1\right\}$$
with $\kappa=\cos\alpha\in(0,1]$ and
$$B_{K,\kappa}^{\rm int}=\left\{x\in\mathbb{R}_+^{n+1}\,|\,x_1\tan\alpha+x_{n+1}>1\right\},$$
where $\alpha\in[0,\frac{\pi}{2})$ (See Figure \ref{equidistant}).

    In this case, set
$$X_2:=x-E_{n+1}, \quad V_2:=\frac{1}{x_{n+1}},\quad Y_2:=x.$$
\begin{center}
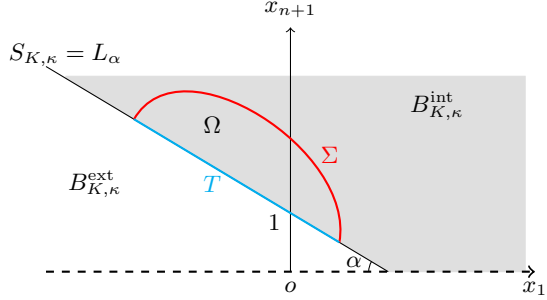
\begin{figure}[h]
\vspace{-10pt}
\begin{tikzpicture}[scale=1.3]
    \draw[gray!25,fill=gray!25] (-2.3,1.98)--(1,0)--(2.4,0)--(2.4,2)--(-2.3,2);
    \draw[dashed,thick,->] (-2.5,0)--(2.5,0);
    \draw[->] (0,0)--(0,2.5);
    \draw (1,0)--(-2.5,2.1);
    \node[left] at (0,0.5) {\footnotesize$1$};\node[below] at (2.5,0) {\footnotesize$x_1$};\node[above] at (0,2.5) {\footnotesize$x_{n+1}$};
    \node[below] at (0,0) {\footnotesize$o$};\node at (-2,0.9) {\footnotesize$B_{K,\kappa}^{\rm ext}$};
    \node at (1.5,1.7) {\footnotesize$B_{K,\kappa}^{\rm int}$};\node at (0.65,0.1) {\footnotesize$\alpha$};
    \draw (0.8,0) arc [radius=0.2,start angle=180,end angle=149.0362];
    \draw[thick,red] (-1.6,1.56) [out=60,in=80] to (0.5,0.3);
    \draw[thick,cyan] (-1.6,1.56)--(0.5,0.3);
    \node at (-0.8,1.5) {\footnotesize$\Omega$};
    \node at (-2.3,2.22) {\footnotesize$S_{K,\kappa}=L_\alpha$};
    \node[cyan] at (-0.8,0.9) {\footnotesize$T$};
    \node[red] at (0.4,1.2) {\footnotesize$\Sigma$};
\end{tikzpicture}
\caption{$S_{K,\kappa}$ is an equidistant hypersurface with $\kappa=\cos\alpha\in(0,1]$, and the shaded area is $B_{K,\kappa}^{\rm int}$.}
\label{equidistant}
\vspace{-15pt}
\end{figure}
\end{center}

\begin{center}
\begin{figure}[h]
\begin{tikzpicture}[scale=1.3]
    \draw[white,fill=gray!25] (2.5,0) arc [radius=2.5,start angle=0,end angle=180];
    \draw[dashed,thick] (0,0) circle [radius=2.5];
    \draw (-2.5,0) to (2.5,0);\node[below] at (0,0) {$o$};\draw[->] (0,0) to (0,2.6);\node[left] at (0,2.6) {\footnotesize{$x_{n+1}$}};
    \draw[thick,red] (-2,0) [out=78,in=105] to (1.6,0);
    \draw[thick,cyan] (-2,0)--(1.6,0);
    \node[below] at (2.2,0) {\footnotesize$S_{K,\kappa}$};
    \node[below,cyan] at (-1.25,0) {\footnotesize$T$};
    \node at (-1,0.6) {\footnotesize$\Omega$};
    \node[red] at (1,0.9) {\footnotesize$\Sigma$};
    \node at (1,1.7) {\footnotesize$B_{K,\kappa}^{\rm int}$};
    \node at (0,-1.3) {\footnotesize$B_{K,\kappa}^{\rm ext}$};
\end{tikzpicture}
\caption{$S_{K,\kappa}$ is a totally geodesic hyperplane with $\kappa=0$, and the shaded area is $B_{K,\kappa}^{\rm int}$.}
\label{hyperplane}
\vspace{-15pt}
\end{figure}
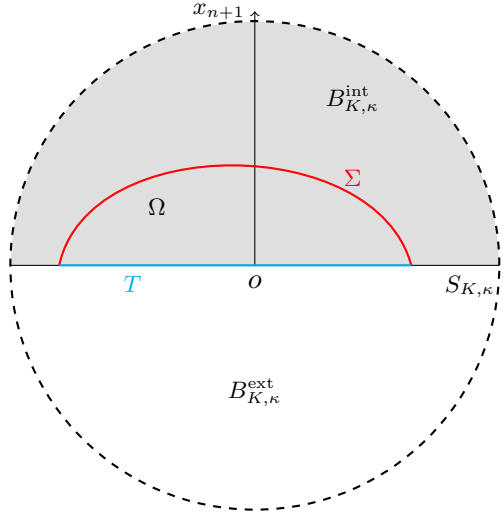
\end{center}

    \textbf{Case 3:} If $S_{K,\kappa}$ is a totally geodesic hyperplane in $\mathbb{H}^{n+1}$, then $\kappa=0$. Using the Poincar\'{e} ball model, we have up to an hyperbolic isometry,
$$S_{K,\kappa}=\left\{x\in\mathbb{B}^{n+1}\,|\,x_{n+1}=0\right\},$$
$$B_{K,\kappa}^{\rm int}=\left\{x\in\mathbb{B}^{n+1}\,|\,x_{n+1}>0\right\}.$$
(See Figure \ref{hyperplane}.)

    In this case, set
\begin{equation}
X_3:=x,\quad V_3:=\frac{1+|x|^2}{1-|x|^2},\quad Y_3:=\frac{1+|x|^2}{2}E_{n+1}-x_{n+1}x.
\label{XVY3}
\end{equation}
    \textbf{Case 4:} If $S_{K,\kappa}$ is a geodesic sphere of radius $R\in (0,\frac{\pi}{2}]$ in $\mathbb{S}^{n+1}$, then $\kappa=\cot R\in [0,\infty)$. In the stereographic projection model, we have
$$S_{K,\kappa}=\left\{x \in \mathbb{B}^{n+1}\, \Bigg| \, |x|= R_{\mathbb{R}}:=\sqrt{\frac{1-\cos R}{1+\cos R}}\right\}.$$ We remark that $S_{K,\kappa}$ is actually a totally geodesic hyperplane ($\kappa=0$) when $R=\frac{\pi}{2}.$ Let $B_{K,\kappa}^{\rm int}$ be the geodesic ball enclosed by $S_{K,\kappa}$, and let $B_{K,\kappa}^{\rm{int},+}$ be the corresponding geodesic half-ball given by
$$B_{K,\kappa}^{\rm{int},+}=\left\{x\in B_{K,\kappa}^{\rm{int}}\,\Big|\,x_{n+1}>0\right\}.$$
(See Figure \ref{SS}.)
\begin{center}
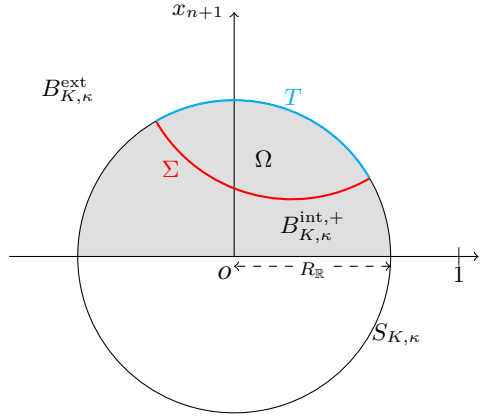
\begin{figure}[h]
\begin{tikzpicture}[scale=1.3]
    \draw[fill=gray!25] (1.6,0) arc [radius=1.6,start angle=0,end angle=180];
    \draw (1.6,0) arc [radius=1.6,start angle=0,end angle=-180];
    \draw[->] (-2.3,0) to (2.5,0);\node[below] at (-0.1,0) {$o$};\draw[->] (0,0) to (0,2.5);\node[left] at (0,2.5) {\footnotesize{$x_{n+1}$}};
    \draw[dashed,<-] (1.6,-0.1) to (1,-0.1);\draw[dashed,->] (0.6,-0.1) to (0,-0.1);\node at (0.8,-0.14) {\tiny{$R_{\mathbb{R}}$}};
    \node at (2.3,0) {\tiny$|$};
    \node[below] at (2.3,0) {\footnotesize$1$};\node at (-1.7,1.7) {\footnotesize$B_{K,\kappa}^{\rm ext}$};
    \draw[thick,red] (1.3856,0.8) to [out=-150,in=-60] (-0.8,1.3856);\draw[thick,cyan] (1.3856,0.8) arc [radius=1.6,start angle=30,end angle=120];
    \node[red] at (-0.65,0.9) {\footnotesize$\Sigma$};\node[cyan] at (0.6,1.63) {\footnotesize$T$};\node at (0.3,1) {\footnotesize$\Omega$};
    \node at (0.8,0.3) {\footnotesize$B_{K,\kappa}^{\rm{int},+}$};
    \node[right] at (1.3,-0.8) {\footnotesize $S_{K,\kappa}$};
\end{tikzpicture}
\caption{$S_{K,\kappa}$ is a geodesic sphere with $\kappa=\cot R\in[0,\infty)$, and the shaded area is $B_{K,\kappa}^{\rm{int},+}$.}
\label{SS}
\vspace{-15pt}
\end{figure}
\end{center}

    In this case, set
\begin{equation*}
     X_4:=\frac{2}{1+R_{\mathbb{R}}^2}\left[x_{n+1}x-\frac{|x|^2+R_{\mathbb{R}}^2}{2}E_{n+1}\right],~V_4:=\frac{2x_{n+1}}{1+|x|^2},~Y_4:=\frac{1-|x|^2}{2}E_{n+1}+x_{n+1}x.
\end{equation*}

    Combining \cite{wang2019uniqueness} and \cite[Proposition 2.2-2.3]{guo2022stable}, we have the following proposition.
\begin{Proposition}
In each case above, the vector fields $X_k$ and $Y_k$ ($k=1,2,3,4$) have the following properties:
\begin{enumerate}
  \item[\rm(1)] $X_k$ is a conformal Killing vector field with $\mathcal{L}_{X_k}\bar{g}=V_k\bar{g}$, namely
$$\frac{1}{2}\left(\bar{g}\big(\overline{\nabla}_{e_i}X_k,e_j\big)+\bar{g}\big(\overline{\nabla}_{e_j}X_k,e_i\big)\right)=V_k\bar{g}_{ij},$$
for any orthonormal frams $\{e_i\}_{i=1}^{n+1}$. In particular, $\overline{\rm div}\,X_k=(n+1)V_k$.
  \item[\rm(2)] $X_k|_{S_{K,\kappa}}$ is a tangential vector field on $S_{K,\kappa}$, namely
\begin{equation}
\bar{g}(X_k,\bar{N})=0 \,\,\text{on} \,\, S_{K,\kappa}.
\label{XorthoN}
\end{equation}
  \item[\rm(3)] $Y_k$ is a Killing vector field, i.e. $\mathcal{L}_{Y_k}\bar{g}=0$. In particular, $\overline{\rm div}\,Y_k=0$.
\end{enumerate}
\label{XYproperty}
\end{Proposition}

    Combining \cite{wang2019uniqueness} and \cite[Proposition 2.3]{guo2022partially}, we have the following proposition.
\begin{Proposition}
In each case above, $V_k$ ($k=1,2,3,4$) satisfies the following properties:
\begin{align}
  \overline{\nabla}^2V_k&=-KV_k\bar{g},\label{V1}\\
  \partial_{\bar{N}}V_k&=\kappa V_k \,\,\text{on}\,\, S_{K,\kappa}.\label{V2}
\end{align}
\end{Proposition}

\section{Existence of auxiliary functions and some integral formulas}\label{section3}
    From this section on, let $\Omega$ be a bounded domain in $B_{K,\kappa}^{\text{int}}$ with piece-wise smooth boundary $\partial\Omega=\overline{\Sigma}\cup T$, where $\Sigma\subset B_{K,\kappa}^{\text{int}}$ is a smooth open hypersurface and $T=\partial\Omega\backslash\overline{\Sigma}\subset S_{K,\kappa}$ meets $\overline{\Sigma}$ at a common $(n-1)$-dimensional submanifold $\Gamma\subset S_{K,\kappa}$. We emphasize here again that: in the following sections, we use $\Omega\subset B_{K,\kappa}^{\rm{int}}$ to indicate that $\Omega\subset B_{K,\kappa}^{\rm{int},+}$ in the case that $S_{K,\kappa}$ is a geodesic sphere.

\subsection{Existence of auxiliary functions}\

    In this subsection, we will find an auxiliary function $\varphi\in C^{\infty}(\mathbb{M}^{n+1}(K))$ for each umbilical hypersurface case which satisfies
\begin{align}\label{resolventfunction}
\left\{
\begin{aligned}
  \overline{\nabla}^2\varphi=\left(\frac{1}{n+1}-K\varphi\right)\bar{g}\quad&\text{on}\,\,\mathbb{M}^{n+1}(K),\\
  \partial_{\bar{N}}\varphi=\kappa \varphi+\widetilde{c} \quad&\text{on}\,\,S_{K,\kappa}.\\
\end{aligned}
\right.
\end{align}

    We first recall the classical $P$-function. For every smooth function $u$, there will be a $P$-function corresponding to it, namely
$$P_u:=\frac{1}{2}\bar{g}\left(\overline{\nabla}u,\overline{\nabla}u\right)-\frac{1}{n+1}u+\frac{K}{2}u^2.$$
If $u$ satisfies $\overline{\triangle}u+(n+1)Ku=1$, from Bochner-Weitzenb\"{o}ck formula we obtain
$$\overline{\triangle}P_u=\left|\overline{\nabla}^2u-\frac{\overline{\triangle}u}{n+1}\bar{g}\right|^2\geq0.$$

    The following lemma is useful for verifying whether a function satisfies the Hessian condition in \eqref{resolventfunction}.
\begin{Lemma}\cite[Lemma 2.1]{ciraolo2019serrin}
    Let $\varphi$ be a solution to
$$\overline{\triangle}\varphi+(n+1)K\varphi=1\quad \text{on} \,\,\mathbb{M}^{n+1}(K),$$
then $\overline{\triangle}P_\varphi=0$ if and only if
$$\overline{\nabla}^2\varphi=\left(\frac{1}{n+1}-K\varphi\right)\bar{g}\quad \text{on} \,\,\mathbb{M}^{n+1}(K).$$
\label{Laplician}
\vspace{-15pt}
\end{Lemma}
    Three space form models, $\mathbb{R}^{n+1}$, $\mathbb{H}^{n+1}$ and $\mathbb{S}^{n+1}$, can be described as the warped product space $I\times\mathbb{S}^n$ equipped with the rotationally symmetric metric $$\bar{g}=dr^2+\psi^2g_{\mathbb{S}^n},$$ where $g_{\mathbb{S}^n}$ is the round metric on $\mathbb{S}^n$ and
\begin{itemize}
  \item $I=[0,\infty)$, $\psi(r)=r$ in the Euclidean case ($K=0$);
  \item $I=[0,\infty)$, $\psi(r)=\sinh r$ in the hyperbolic case ($K=-1$);
  \item $I=[0,\pi)$, $\psi(r)=\sin r$ in the spherical case ($K=1$).
\end{itemize}

    We just consider $K=\pm 1$ in the following. One can easily find that
\begin{equation}
\ddot{\psi}=-K\psi,\quad \dddot{\psi}=-K\dot{\psi}.
\label{psi}
\end{equation}
Enlightened by \cite{ciraolo2019serrin}, we know the function
\begin{equation}
\varphi(r):=c_0\dot{\psi}(r)+\frac{K}{n+1}.
\end{equation}
will satisfies $\overline{\triangle}\varphi+(n+1)K\varphi=1$ since
\begin{equation}
\overline{\triangle}\varphi=\ddot{\varphi}+n\dot{\psi}\psi^{-1}\dot{\varphi}=-c_0(n+1)K\dot{\psi},
\end{equation}
where $c_0$ is an undetermined constant for $\varphi$ to satisfy the boundary condition $\partial_{\bar{N}}\varphi=\kappa\varphi+\widetilde{c}$ on $S_{K,\kappa}$.

    Using \eqref{psi}, we get $\bar{g}(\overline{\nabla}\varphi,\overline{\nabla}\varphi)=c_0^2(\ddot{\psi})^2=c_0^2\psi^2$. Compute the Laplacian of $P_\varphi$ directly,
\begin{align*}
    \overline{\triangle}P_\varphi&=\frac{1}{2}\overline{\triangle}\bar{g}\left(\overline{\nabla}\varphi,\overline{\nabla}\varphi\right)-\frac{1}{n+1}\overline{\triangle}\varphi
+K\bar{g}\left(\overline{\nabla}\varphi,\overline{\nabla}\varphi\right)+K\varphi\overline{\triangle}\varphi\\
&=c_0^2\left(\dot{\psi}\right)^2+c_0^2\psi\overline{\triangle}\psi+c_0K\dot{\psi}+c_0^2K\psi^2+K\left(c_0\dot{\psi}+\frac{K}{n+1}\right)\left(-c_0(n+1)K\dot{\psi}\right)\\
&=c_0^2\left(\dot{\psi}\right)^2+c_0^2\psi\left(\ddot{\psi}+n\dot{\psi}\psi^{-1}\dot{\psi}\right)+c_0^2K\psi^2-c_0^2(n+1)\left(\dot{\psi}\right)^2\\
&=c_0^2\left(\dot{\psi}\right)^2-c_0^2K\psi^2+c_0^2n\left(\dot{\psi}\right)^2+c_0^2K\psi^2-c_0^2(n+1)\left(\dot{\psi}\right)^2\\
&=0,
\end{align*}
where we have used $K^2=1$ and \eqref{psi} in the computation. Hence, from Lemma \ref{Laplician} we know $\varphi$ satisfies $$\overline{\nabla}^2\varphi=\left(\frac{1}{n+1}-K\varphi\right)\bar{g}\quad\text{on}\,\,\mathbb{M}^{n+1}(K).$$

    Now, for each umbilical hypersurface case, we will determine the constant $c_0$ such that $\varphi$ satisfies the boundary condition $\partial_{\bar{N}}\varphi=\kappa\varphi+\widetilde{c}$ on $S_{K,\kappa}$.

    \textbf{Case 1:} $S_{K,\kappa}$ is a geodesic sphere of radius $R$ in $\mathbb{H}^{n+1}$ ($K=-1$), $\kappa=\coth R\in (1,\infty)$. Note that $\bar{N}=\frac{\partial}{\partial r}$, we can calculate directly as follows,
\begin{align*}
    \left(\partial_{\bar{N}}\varphi-\kappa\varphi\right)\big|_{S_{K,\kappa}}&=\dot{\varphi}(R)-\kappa\varphi(R)\\
    &=c_0\ddot{\psi}(R)-\kappa\left(c_0\dot{\psi}(R)-\frac{1}{n+1}\right)\\
    &=c_0\psi(R)-c_0\kappa\dot{\psi}(R)+\frac{\kappa}{n+1}\\
    &=c_0\left(\psi(R)-\kappa\dot{\psi}(R)\right)+\frac{\kappa}{n+1}\\
    &=c_0(1-\kappa^2)\sinh R+\frac{\kappa}{n+1},
\end{align*}
where we have used \eqref{psi} and $\kappa=\coth R$. Since $\sinh R>0, 1-\kappa^2\neq 0$, we can determine $c_0$ such that $\partial_{\bar{N}}\varphi=\kappa\varphi+\widetilde{c}$ on $S_{K,\kappa}$.

    \textbf{Case 2:} $S_{K,\kappa}=L_\alpha$ ($\alpha\in[0,\frac{\pi}{2})$) is an equidistant hypersurface in $\mathbb{H}^{n+1}$ ($K=-1$), $\kappa=\cos\alpha\in (0,1]$. To make it easier to differentiate on $S_{K,\kappa}$, we map $S_{K,\kappa}=L_\alpha$ to the Poincar\'{e} ball model (See Figure \ref{transform}). The formula for the coordinate transformation from the upper half-space model to the Poincar\'{e} ball model is \cite[Section 3.4]{beardon2012geometry}:
\begin{equation*}
\phi(x)=\phi_{E_{n+1},\sqrt{2}}\circ\sigma(x)=\frac{\left(2x_1,\cdots,2x_n,|x|^2-1\right)}{x_1^2+\cdots+x_n^2+(x_{n+1}+1)^2},
\end{equation*}
where $\phi_{E_{n+1},\sqrt{2}}$ is the reflection of the sphere $\mathbb{S}^n(E_{n+1},\sqrt{2})$ and $\sigma$ is the reflection of the plane $x_{n+1}=0$, namely
\begin{align*}
\phi_{E_{n+1},\sqrt{2}}(x)&=E_{n+1}+\frac{2(x-E_{n+1})}{|x-E_{n+1}|^2},\\
\sigma(x)&=(x_1,\cdots,x_n,-x_{n+1}).
\end{align*}
Then $\phi$ maps $L_\alpha\subset\mathbb{R}_+^{n+1}$ to
$$L_\alpha^\mathbb{B}=\left\{x\in\mathbb{B}^{n+1}\,\Bigg|\,\left(x_1-\frac{\tan\alpha}{2}\right)^2+x_2^2+\cdots x_n^2+\left(x_{n+1}-\frac{1}{2}\right)^2=\frac{\sec^2\alpha}{4}\right\}.$$
Now in $\left(\mathbb{B}^{n+1},\bar{g}=\frac{4}{(1-|x|^2)^2}\delta\right)$, $\varphi=c_0\frac{1+|x|^2}{1-|x|^2}-\frac{1}{n+1}$ since $r=\ln\frac{1+|x|}{1-|x|}$. Note that $$\bar{N}=\cos\alpha\left(1-|x|^2\right)\left[x-\frac{\tan\alpha}{2}E_1-\frac{1}{2}E_{n+1}\right] \text{on} \,\,L_\alpha^\mathbb{B},$$ then
\begin{align*}
    \left(\partial_{\bar{N}}\varphi-\kappa\varphi\right)\big|_{S_{K,\kappa}}&=\left(\bar{N}^i\frac{\partial\varphi}{\partial x_i}-\kappa\varphi\right)\Bigg|_{L_\alpha^\mathbb{B}}\\
    &=\frac{4c_0\cos\alpha}{1-|x|^2}\left(|x|^2-\frac{\tan\alpha}{2}x_1-\frac{1}{2}x_{n+1}\right)-\cos\alpha\left(c_0\frac{1+|x|^2}{1-|x|^2}-\frac{1}{n+1}\right)\\
    &=-c_0\cos\alpha+\frac{\cos\alpha}{n+1},
\end{align*}
where we have used $|x|^2-x_1\tan\alpha-x_{n+1}=0$ on $L_\alpha^\mathbb{B}$. Hence we can determine $c_0$ such that $\partial_{\bar{N}}\varphi=\kappa\varphi+\widetilde{c}$ on $S_{K,\kappa}$ since $\cos\alpha\neq 0$.
\begin{center}
\begin{figure}[h]
\begin{tikzpicture}[scale=1]
    \draw[white,fill=gray!25] (3,2) arc [radius=2,start angle=90,end angle=-47];
    \draw[white,fill=gray!25] (-5.8,2)--(-2.5,0)--(-2,0)--(-2,2);
    \draw[dashed,thick,->] (-6,0)--(-2,0);\node[below] at (-2,0) {\footnotesize$x_1$};
    \draw[->] (-4,0)--(-4,2.5);\node[left] at (-4,2.55) {\footnotesize$x_{n+1}$};
    \node[below] at (-4,0) {\footnotesize$o$};
    \draw[thick,red] (-2.5,0)--(-5.8,2);\node[red,above] at (-6,1.9) {\footnotesize$S_{K,\kappa}=L_\alpha$};\node[left] at (-4,0.82) {\footnotesize$1$};
    \draw (-2.8,0) arc [radius=0.25,start angle=180,end angle=142];\node[above] at (-3,-0.08) {\footnotesize$\alpha$};
    \draw[->] (-5.2,1.6364)--(-5.5,1.1414);\node[below] at (-5.5,1.1414) {\footnotesize$\bar{N}$};
    \node[left] at (3,2.55) {\footnotesize$x_{n+1}$};
    \draw[dashed,thick] (3,0) circle [radius=2];\draw[fill=black] (5.5,1) circle [radius=0.02];
    \draw[dashed] (5.5,1) circle [radius=2.6926];
    \draw[->] (-1.7,2)--(1,2);\node[above] at (0,2) {$\phi$};
    \node[below] at (2.9,0) {\footnotesize$o$};
    \draw[thick,red,fill=gray!25] (3,2) arc [radius=2.6926,start angle=158,end angle=245.5];
    \draw[->] (3,0)--(3,2.5);
    \draw[->] (3,0)--(7.3,0);\node[below] at (7.3,0) {\footnotesize$x_1$};
    \node[right] at (2.9,1.8) {\footnotesize$1$};
    \node[right] at (5.5,1) {\footnotesize$(\frac{\tan\alpha}{2},0,\cdots,0,\frac{1}{2})$};
    \draw (3,0)--(5.5,1);\node[red] at (2.5,1) {\footnotesize$L_{\alpha}^{\mathbb{B}}$};
    \node[above] at (4,0.3) {\footnotesize$\frac{\sec\alpha}{2}$};
    \draw[->] (3.6,-0.9)--(3.2,-1.3);\draw[dashed] (3.6,-0.9)--(5.5,1);\node[below] at (3.2,-1.3) {\footnotesize$\bar{N}$};
\end{tikzpicture}
\caption{The transformation from the upper half-space model to the Poincar\'{e} ball model.}
\label{transform}
\end{figure}
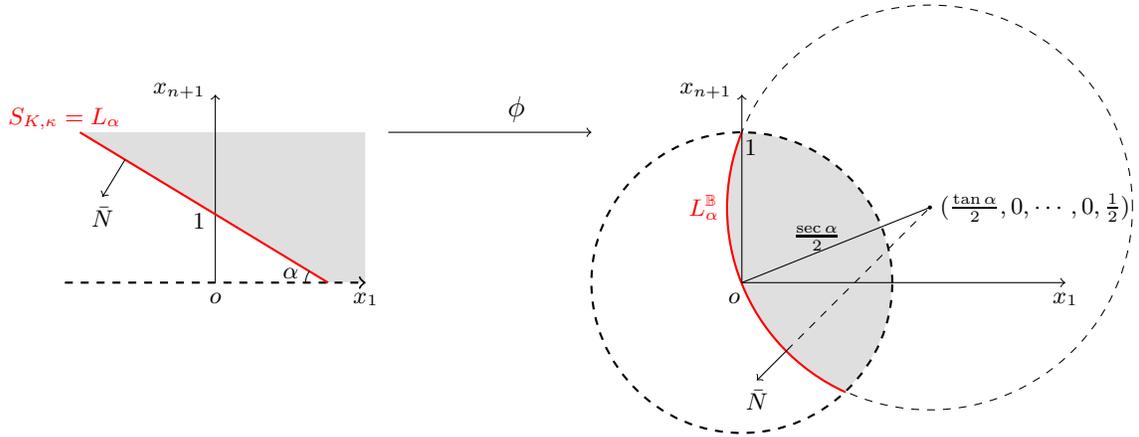
\end{center}

    \textbf{Case 3:} $S_{K,\kappa}$ is a totally geodesic hyperplane in $\mathbb{H}^{n+1}$ ($K=-1$), $\kappa=0$. In this case, we need to change the distance base point, otherwise it will happen that $\partial_{\bar{N}}\varphi=0$, which will make $c_0$ unable to determined. Precisely, we let $$\varphi=\dot{\psi}\left(\bar{d}(x,x_{c_0})\right)+\frac{K}{n+1},$$
where $x_{c_0}=(0,\cdots,0,c_0)\in\mathbb{B}^{n+1}$ with $|c_0|<1$, and $\bar{d}(x,x_{c_0})$ is the distance of $x, x_{c_0}\in\mathbb{B}^{n+1}$ under the Poincar\'{e} metric, then $\varphi$ still satisfy
\begin{equation*}
  \overline{\nabla}^2\varphi=\left(\frac{1}{n+1}-K\varphi\right)\bar{g}\quad\text{on} \;\mathbb{M}^{n+1}(K),
\end{equation*}
since isometry reason.
According to the formula for the distance between any two points in the Poincar\'{e} ball model \cite[Exercise 3.4]{beardon2012geometry}, we know
$$\varphi=\frac{1+c_0^2}{1-c_0^2}\frac{1+|x|^2}{1-|x|^2}-\frac{4c_0}{1-c_0^2}\frac{x_{n+1}}{1-|x|^2}-\frac{1}{n+1}.$$
Note that $\bar{N}=-\frac{1-|x|^2}{2}E_{n+1}$, thus
\begin{align*}
\left(\partial_{\bar{N}}\varphi-\kappa\varphi\right)\big|_{S_{K,\kappa}}&=-\frac{1-|x|^2}{2}\frac{\partial\varphi}{\partial x_{n+1}}\\
&=-\frac{1-|x|^2}{2}\left(\frac{1+c_0^2}{1-c_0^2}\frac{4x_{n+1}}{(1-|x|^2)^2}-\frac{4c_0}{1-c_0^2}\frac{1-|x|^2+2x_{n+1}^2}{(1-|x|^2)^2}\right)\\
&=\frac{2c_0}{1-c_0^2},
\end{align*}
where we have used $x_{n+1}=0$ on $S_{K,\kappa}$. Take $c_0=\frac{\widetilde{c}}{1+\sqrt{1+\widetilde{c}^2}}$, then $\partial_{\bar{N}}\varphi=\kappa\varphi+\widetilde{c}$ on $S_{K,\kappa}$.

    \textbf{Case 4:} $S_{K,\kappa}$ is a geodesic sphere of radius $R\in(0,\frac{\pi}{2}]$ in $\mathbb{S}^{n+1}$ ($K=1$), $\kappa=\cot R\geq 0$. Note that $\bar{N}=\frac{\partial}{\partial r}$, we can calculate directly as follows,
\begin{align*}
\left(\partial_{\bar{N}}\varphi-\kappa\varphi\right)\big|_{S_{K,\kappa}}&=c_0\ddot{\psi}(R)-\kappa\left(c_0\dot{\psi}(R)+\frac{1}{n+1}\right)\\
&=-c_0\left(\psi(R)+\kappa\dot{\psi}(R)\right)-\frac{\kappa}{n+1}\\
&=-c_0\sin R(1+\kappa^2)-\frac{\kappa}{n+1},
\end{align*}
where we have used \eqref{psi} and $\kappa=\cot R$. Since $\sin R\neq 0$, $1+\kappa^2\neq 0$, we can determine $c_0$ such that $\partial_{\bar{N}}\varphi=\kappa\varphi+\widetilde{c}$ on $S_{K,\kappa}$.

    So far, we have obtained the following proposition.
\begin{Proposition}
For each umbilical hypersurface $S_{K,\kappa}$ in $\mathbb{M}^{n+1}(K)$ ($K=\pm 1$), there exists an auxiliary function $\varphi\in C^\infty(\mathbb{M}^{n+1}(K))$ that satisfies
\begin{align}
\left\{
\begin{aligned}
  \overline{\nabla}^2\varphi=\left(\frac{1}{n+1}-K\varphi\right)\bar{g}\quad&\text{on}\,\,\mathbb{M}^{n+1}(K),\\
  \partial_{\bar{N}}\varphi=\kappa \varphi+\widetilde{c} \quad&\text{on}\,\,S_{K,\kappa}.\\
\end{aligned}
\right.
\end{align}
\label{existphi}
\end{Proposition}

\subsection{Some integral formulas}\

    In this subsection, we will give some basic integral formulas and the Minkowski formulas, and then prove Proposition \ref{meancurvature} about the mean curvature which is useful in the contact angle characterization.
\begin{Lemma}
For each umbilical hypersurface case $k=1,2,3,4$, there hold
\begin{align}
\int_\Sigma\bar{g}(\nu,Y_k)dA&=-\frac{1}{\mathcal{M}_k}\int_TV_kdA,\label{div1}\\
\int_\Sigma H\bar{g}(\nu,Y_k)dA&=-\int_\Gamma\bar{g}(\mu,Y_k)ds\label{div2}.
\end{align}
See \eqref{Minkowski} for the definition of $\mathcal{M}_k$.
\end{Lemma}
\begin{proof}
Since $\overline{\rm{div}}\,Y_k=0$, by using the divergence theorem in $\Omega$, we have
$$0=\int_\Omega\overline{\rm{div}}\,Y_kdx=\int_\Sigma\bar{g}(\nu,Y_k)dA+\int_T\bar{g}(\bar{N},Y_k)dA.$$
A direct computation shows that $\bar{g}(\bar{N},Y_k)=\frac{1}{\mathcal{M}_k}V_k$, hence
$$\int_\Sigma\bar{g}(\nu,Y_k)dA=-\int_T\bar{g}(\bar{N},Y_k)dA=-\frac{1}{\mathcal{M}_k}\int_TV_kdA.$$

Let $\{e_i\}_{i=1}^n$ be any orthonormal frame on $\Sigma$, then
\begin{align*}
\text{div}^\Sigma(Y_k^t) &= \sum_{i=1}^n\bar{g}(\nabla_{e_i}^\Sigma Y_k^t,e_i) \\
&= \sum_{i=1}^n\bar{g}(\overline{\nabla}_{e_i}Y_k^t,e_i) \\
&= \sum_{i=1}^n\left[e_i\bar{g}(Y_k^t,e_i)-\bar{g}(Y_k^t,\overline{\nabla}_{e_i}e_i)\right] \\
&= \sum_{i=1}^n\left[e_i\bar{g}(Y_k,e_i)-\bar{g}(Y_k,\overline{\nabla}_{e_i}e_i)+\bar{g}(\nu,Y_k)\bar{g}(\nu,\overline{\nabla}_{e_i}e_i)\right] \\
&= \sum_{i=1}^n\bar{g}(\overline{\nabla}_{e_i}Y_k,e_i)+\bar{g}(\nu,Y_k)\sum_{i=1}^n\bar{g}(\nu,\overline{\nabla}_{e_i}e_i), \\
&= -H\bar{g}(\nu,Y_k),
    \end{align*}
where we have used $Y_k$ is a Killing vector field in the last equality. By using the divergence theorem on $\Sigma$, we have
$$\int_\Sigma H\bar{g}(\nu,Y_k)dA=-\int_\Sigma\text{div}^\Sigma(Y_k^t)dA=-\int_\Gamma\bar{g}(\mu,Y_k^t)ds=-\int_\Gamma\bar{g}(\mu,Y_k)ds,$$
hence we prove \eqref{div2}.
\end{proof}

    If a hypersurface $\Sigma$ meets an umbilical hypersurface at a constant angle $\theta\in(0,\pi)$, a Minkowski formula will hold by standard argument. In \cite{wang2019uniqueness}, the authors give a Minkowski formula for hypersurface $\Sigma$ support on a geodesic sphere in a space form $\mathbb{M}^{n+1}(K)$ ($K=0,-1,1$). Minkowski formulas for hypersurface $\Sigma$ supported on a totally geodesic hyperplane and equidsitant hypersurface are given in \cite{chen2022some}. For the convenience of readers, we state them in the following proposition.
\begin{Proposition}[Minkowski formulas]
If $\Sigma$ meets $S_{K,k}$ at a constant angle $\theta\in(0,\pi)$, then
\begin{equation}
\int_\Sigma n\left(V_k+\mathcal{M}_k\cos\theta\,\bar{g}(\nu,Y_k)\right)dA=\int_\Sigma H\bar{g}(\nu,X_k)dA,
\label{minkowskiformula}
\end{equation}
where
\begin{align}
\mathcal{M}_k:=\left\{
\begin{aligned}
\sinh R, \,\,k&=1,\\
-\sec\alpha, \,\,k&=2,\\
-1, \,\,k&=3,\\
\sin R, \,\,k&=4.
\end{aligned}
\right.
\label{Minkowski}
\end{align}
\end{Proposition}

\begin{Remark}
 For simplicity, for each umbilical hypersurface case, we will directly use $X$, $V$, $Y$, $\mathcal{M}$ instead of $X_k$, $V_k$, $Y_k$, $\mathcal{M}_k$ in the following statement.
\end{Remark}

\begin{Proposition}
Given $\theta\in(0,\pi)$, if $\Sigma$ is a CMC hypersurface meeting $S_{K,\kappa}$ at a fixed contact angle $\theta$, then
\begin{equation}\label{meancurvature}
H\left(\int_\Omega Vdx+\frac{n}{n+1}\cos\theta\frac{(\int_T VdA)^2}{\mathcal{M}\int_\Gamma\bar{g}(\mu,Y)ds}\right)=\frac{n}{n+1}\int_\Sigma VdA.
\end{equation}
\end{Proposition}
\begin{proof}
A simple integration by parts gives
\begin{equation}
(n+1)\int_\Omega Vdx=\int_\Omega\overline{\text{div}}\,Xdx=\int_\Sigma\bar{g}(\nu,X)dA,
\end{equation}
where we have used property \eqref{XorthoN} in Proposition \ref{XYproperty} in the second equality. Since $H$ is a constant, from \eqref{div1} and the Minkowski formula \eqref{minkowskiformula}, we have
\begin{align}
\begin{aligned}
\int_\Sigma VdA&=\frac{H}{n}\int_\Sigma\bar{g}(\nu,X)dA-\mathcal{M}\cos\theta\int_\Sigma\bar{g}(\nu,Y)dA\\
&=\frac{n+1}{n}H\int_\Omega Vdx+\cos\theta\int_TVdA,
\end{aligned}
\label{intV}
\end{align}
Since $H$ is a constant again, \eqref{div1} and \eqref{div2} imply $$H=\frac{\mathcal{M}\int_\Gamma\bar{g}(\mu,Y)ds}{\int_TVdA},$$
thus we may rewrite the second term on the right of \eqref{intV} to be
$$\cos\theta\int_TVdA=\cos\theta\frac{\left(\int_TVdA\right)^2}{\mathcal{M}\int_\Gamma\bar{g}(\mu,Y)ds}H,$$
which in turn gives that
\begin{align*}
\int_\Sigma VdA&=\frac{n+1}{n}H\left(\int_\Omega Vdx+\frac{n}{n+1}\cos\theta\frac{\left(\int_TVdA\right)^2}{\mathcal{M}\int_\Gamma\bar{g}(\mu,Y)ds}\right)\\
&=\frac{n+1}{n}H\left(\int_\Omega Vdx-c(\theta,\Omega)\int_T VdA\right).
\end{align*}
\end{proof}

\section{Integral identities and proof of Theorem \ref{Theorem}}\label{section4}
    The regularity assumption $u\in W^{1,\infty}(\Omega)\cap W^{2,2}(\Omega)$ in Theorem \ref{Theorem} is for technical reason, that is, we will use an integration method which requires to perform integration by parts.
\begin{Lemma}\cite[Lemma 2.1]{pacella2020overdetermined}
Let $F: \Omega\rightarrow\mathbb{R}^{n+1}$ be a vector field such that
$$F\in C^1(\overline{\Omega}\,\backslash\,\Gamma)\cap L^2(\Omega)\,\,\text{and}\,\,\overline{\rm div}\,F\in L^1(\Omega),$$
then $$\int_\Omega\overline{\rm div}\,Fdx=\int_\Sigma\bar{g}(F,\nu)dA+\int_T\bar{g}(F,\bar{N})dA.$$
\end{Lemma}
\begin{Remark}
With this lemma, all of the following integration by parts formulas can be performed since $u\in C^\infty(\overline{\Omega}\,\backslash\,\Gamma)\cap W^{1,\infty}(\Omega)\cap W^{2,2}(\Omega)$.
\end{Remark}
\begin{Proposition}
Let $Z$ be a tangent vector field to $T$, $u$ be the unique solution to \eqref{BVP}, then
\begin{equation}
\overline{\nabla}^2u\left(\bar{N},Z\right)=0\,\,\text{on}\,\,T.
\label{Hessianu}
\end{equation}
\end{Proposition}
\begin{proof}
By differentiating both sides of the equation $\partial_{\bar{N}}u=\kappa u+\widetilde{c}$ with respect to $Z$, and note that $\widetilde{c}$ is a constant, we get
\begin{align*}
\kappa Z(u)&=Z\bar{g}(\overline{\nabla}u,\bar{N})\\
&=\overline{\nabla}^2u(\bar{N},Z)+\bar{g}(\overline{\nabla}u,\overline{\nabla}_Z\bar{N})\\
&=\overline{\nabla}^2u(\bar{N},Z)+h^{S_{K,\kappa}}(\overline{\nabla}u,Z)\\
&=\overline{\nabla}^2u(\bar{N},Z)+\kappa\bar{g}(\overline{\nabla}u,Z),
\end{align*}
where we have used the fact that $S_{K,\kappa}$ is an umbilical hypersurface with principal curvature $\kappa$. Hence, the assertion \eqref{Hessianu} follows.
\end{proof}
    The following is a crucial integral identity in the proof of Theorem \ref{Theorem}. For brevity, we use $P$ to refer to $P_u$, and we uniformly denote the outward normals of $\Sigma$ and $T$ as $\nu$ when we integrate over the entire boundary $\partial\Omega=\overline{\Sigma}\cup T$.
\begin{Lemma}
 Let $u\in W^{1,\infty}(\Omega)\cap W^{2,2}(\Omega)$ be a weak solution to the mixed BVP \eqref{BVP}. Then we have: for any constant $a\in\mathbb{R}$, there holds
\begin{align}
\begin{aligned}
&\int_\Omega-Vu\,\overline{\triangle}Pdx\\
=&\frac{1}{2}\int_\Sigma\left(\bar{g}(\overline{\nabla}u,\overline{\nabla}u)-a\right)\left[V(u-\varphi)_\nu-V_\nu(u-\varphi)\right]dA.
\label{arbitrary}
\end{aligned}
\end{align}
\label{integralidentity}
\end{Lemma}
\begin{proof}
Integrating by parts, we have
\begin{equation*}
\int_\Omega-Vu\,\overline{\triangle}Pdx=-\int_{\partial\Omega}VuP_\nu dA+\int_{\partial\Omega}P(Vu)_\nu dA-\int_\Omega P\,\overline{\triangle}(Vu)dx.
\end{equation*}
Using $\overline{\triangle}u+(n+1)Ku=1$ and \eqref{V1}, we have
\begin{align*}
\overline{\triangle}(Vu)&=V\,\overline{\triangle}u+u\,\overline{\triangle}V+2\bar{g}(\overline{\nabla}V,\overline{\nabla}u)\\
&=V+2\bar{g}(\overline{\nabla}V,\overline{\nabla}u)-2(n+1)KVu.
\end{align*}
Thus,
\begin{align}\label{thus}
\begin{aligned}
&\int_\Omega-Vu\,\overline{\triangle}Pdx\\
=&-\int_{\partial\Omega}VuP_\nu dA+\int_{\partial\Omega}P(Vu)_\nu dA-\int_\Omega VPdx-2\int_\Omega P\bar{g}(\overline{\nabla}V,\overline{\nabla}u)dx\\
 &+2(n+1)K\int_\Omega VuPdx.
\end{aligned}
\end{align}
Keeping track of the second term, by using an auxiliary function $\varphi$ which exists thanks to Proposition \ref{existphi}, we compute as follows:
\begin{align}
\begin{aligned}
\int_{\partial\Omega}P(Vu)_\nu dA=&\int_{\partial\Omega}P(V(u-\varphi))_\nu dA+\int_{\partial\Omega}P(V\varphi)_\nu dA\\
=&\int_{\partial\Omega}VP(u_\nu-\varphi_\nu)dA+\int_{\partial\Omega}uPV_\nu dA+\int_{\partial\Omega}VP\varphi_\nu dA\\
=&\int_{\partial\Omega}VP(u_\nu-\varphi_\nu)dA+\int_\Omega\overline{\text{div}}(uP\,\overline{\nabla}V)dx+\int_{\partial\Omega}VP\varphi_\nu dA\\
=&\int_{\partial\Omega}VP(u_\nu-\varphi_\nu)dA+\int_\Omega P\bar{g}(\overline{\nabla}V,\overline{\nabla}u)dx+\int_\Omega u\bar{g}(\overline{\nabla}V,\overline{\nabla}P)dx\\
&+\int_\Omega uP\,\overline{\triangle}Vdx+\int_{\partial\Omega}VP\varphi_\nu dA\\
=&\int_{\partial\Omega}VP(u_\nu-\varphi_\nu)dA+\int_\Omega P\bar{g}(\overline{\nabla}V,\overline{\nabla}u)dx+\int_\Omega u\bar{g}(\overline{\nabla}V,\overline{\nabla}P)dx\\
&+\int_{\partial\Omega}VP\varphi_\nu dA-(n+1)K\int_\Omega VuPdx,
\end{aligned}\label{secondterm}
\end{align}
where we have used \eqref{V1} in the last equality. Using integration by parts, we have
\begin{align}
\begin{aligned}
\int_{\partial\Omega}VP\varphi_\nu dA=&\int_\Omega VP\,\overline{\triangle}\varphi dx+\int_\Omega P\bar{g}(\overline{\nabla}V,\overline{\nabla}\varphi)dx+\int_\Omega V\bar{g}(\overline{\nabla}P,\overline{\nabla}\varphi)dx\\
=&\int_\Omega VPdx+\int_\Omega P\bar{g}(\overline{\nabla}V,\overline{\nabla}\varphi)dx+\int_\Omega V\bar{g}(\overline{\nabla}P,\overline{\nabla}\varphi)dx\\
 &-(n+1)K\int_\Omega VP\varphi dx,
\end{aligned}\label{expanding}
\end{align}
where we have used $\overline{\triangle}\varphi+(n+1)K\varphi=1$ in the last equality, then expanding $P$ in the second-to-last term in \eqref{expanding}, we get
\begin{align}
\begin{aligned}
\int_{\partial\Omega}VP\varphi_\nu dA=&\int_\Omega VPdx+\int_\Omega P\bar{g}(\overline{\nabla}V,\overline{\nabla}\varphi)dx+\int_\Omega V\overline{\nabla}^2u(\overline{\nabla}u,\overline{\nabla}\varphi)dx-\frac{1}{n+1}\int_\Omega V\bar{g}(\overline{\nabla}u,\overline{\nabla}\varphi)dx\\
 &+K\int_\Omega Vu\bar{g}(\overline{\nabla}u,\overline{\nabla}\varphi)dx-(n+1)K\int_\Omega VP\varphi dx.
\end{aligned}\label{afterexpanding}
\end{align}
By using the Ricci identity, $\overline{\triangle}u+(n+1)Ku=1$ and $\overline{\text{Ric}}=nK\bar{g}$, we have
\begin{align}
\begin{aligned}
&\int_\Omega V\overline{\nabla}^2u(\overline{\nabla}u,\overline{\nabla}\varphi)dx\\
=&\int_{\partial\Omega}Vu\overline{\nabla}^2u(\overline{\nabla}\varphi,\nu)dA-\int_\Omega u\overline{\nabla}^2u(\overline{\nabla}V,\overline{\nabla}\varphi)dx\\
&-\int_\Omega Vu\bar{g}(\overline{\nabla}\,\overline{\triangle}u,\overline{\nabla}\varphi)dx-\int_\Omega Vu\,\overline{\text{Ric}}(\overline{\nabla}u,\overline{\nabla}\varphi)dx-\int_\Omega Vu\bar{g}(\overline{\nabla}^2u,\overline{\nabla}^2\varphi)dx\\
=&\int_{\partial\Omega}Vu\overline{\nabla}^2u(\overline{\nabla}\varphi,\nu)dA-\int_\Omega u\overline{\nabla}^2u(\overline{\nabla}V,\overline{\nabla}\varphi)dx-\int_\Omega Vu\bar{g}(\overline{\nabla}^2u,\overline{\nabla}^2\varphi)dx+K\int_\Omega Vu\bar{g}(\overline{\nabla}u,\overline{\nabla}\varphi)dx.
\end{aligned}\label{usingRici}
\end{align}
Since $\varphi$ satisfies \eqref{resolventfunction}, we compute as follows:
\begin{align}
\begin{aligned}
&\int_\Omega Vu\bar{g}(\overline{\nabla}^2u,\overline{\nabla}^2\varphi)dx\\
=&\int_\Omega Vu\left(\frac{1}{n+1}-K\varphi\right)\overline{\triangle}udx\\
=&\frac{1}{n+1}\int_\Omega Vu\,\overline{\triangle}\varphi(1-(n+1)Ku)dx\\
=&\frac{1}{n+1}\int_\Omega Vu\,\overline{\triangle}\varphi dx-K\int_\Omega Vu^2\,\overline{\triangle}\varphi dx\\
=&\frac{1}{n+1}\int_{\partial\Omega} Vu\varphi_\nu dA-\frac{1}{n+1}\int_\Omega V\bar{g}(\overline{\nabla}u,\overline{\nabla}\varphi)dx-\frac{1}{n+1}\int_\Omega u\bar{g}(\overline{\nabla}V,\overline{\nabla}\varphi)dx\\
&-K\int_{\partial\Omega}Vu^2\varphi_\nu dA+2K\int_\Omega Vu\bar{g}(\overline{\nabla}u,\overline{\nabla}\varphi)dx+K\int_\Omega u^2\bar{g}(\overline{\nabla}V,\overline{\nabla}\varphi)dx,
\end{aligned}\label{2parts}
\end{align}
where we use integration by parts in the last equality. Substituting \eqref{2parts} into \eqref{usingRici}, then \eqref{afterexpanding} becomes
\begin{align*}
\begin{aligned}
\int_{\partial\Omega}VP\varphi_\nu dA=&\int_\Omega VPdx+\int_\Omega P\bar{g}(\overline{\nabla}V,\overline{\nabla}\varphi)dx+\int_{\partial\Omega}Vu\overline{\nabla}^2u(\overline{\nabla}\varphi,\nu)dA\\
&-\int_\Omega u\overline{\nabla}^2u(\overline{\nabla}V,\overline{\nabla}\varphi)dx+\frac{1}{n+1}\int_\Omega u\bar{g}(\overline{\nabla}V,\overline{\nabla}\varphi)dx-\frac{1}{n+1}\int_{\partial\Omega}Vu\varphi_\nu dA\\
&+K\int_{\partial\Omega}Vu^2\varphi_\nu dA-K\int_\Omega u^2\bar{g}(\overline{\nabla}V,\overline{\nabla}\varphi)dx-(n+1)K\int_\Omega VP\varphi dx,
\end{aligned}
\end{align*}
substituting this into \eqref{secondterm}, we can finally go back to \eqref{thus} to obtain
\begin{align}
\begin{aligned}
&\int_\Omega-Vu\,\overline{\triangle}Pdx\\
=&-\int_{\partial\Omega}VuP_\nu dA+\int_{\partial\Omega}VP(u_\nu-\varphi_\nu)dA-\int_\Omega P\bar{g}(\overline{\nabla}V,\overline{\nabla}(u-\varphi))dx+\int_\Omega u\bar{g}(\overline{\nabla}V,\overline{\nabla}P)dx\\
&+\int_{\partial\Omega} Vu\overline{\nabla}^2u(\overline{\nabla}\varphi,\nu)dA-\int_\Omega u\overline{\nabla}^2u(\overline{\nabla}V,\overline{\nabla}\varphi)dx+\frac{1}{n+1}\int_\Omega u\bar{g}(\overline{\nabla}V,\overline{\nabla}\varphi)dx-\frac{1}{n+1}\int_{\partial\Omega} Vu\varphi_\nu dA\\
&+K\int_{\partial\Omega} Vu^2\varphi_\nu dA-K\int_\Omega u^2\bar{g}(\overline{\nabla}V,\overline{\nabla}\varphi)dx+(n+1)K\int_\Omega VP(u-\varphi)dx.
\end{aligned}\label{beforeI}
\end{align}
We collect the sum of the integrals over $\Omega$ without $K$ in \eqref{beforeI}, which is
\begin{align*}
\Lambda=&-\int_\Omega P\bar{g}(\overline{\nabla}V,\overline{\nabla}(u-\varphi))dx+\int_\Omega u\bar{g}(\overline{\nabla}V,\overline{\nabla}P)dx-\int_\Omega u\overline{\nabla}^2u(\overline{\nabla}V,\overline{\nabla}\varphi)dx\\
&+\frac{1}{n+1}\int_\Omega u\bar{g}(\overline{\nabla}V,\overline{\nabla}\varphi)dx,
\end{align*}
then expanding $P$, we get
\begin{align}
\begin{aligned}
\Lambda=&-\frac{1}{2}\int_\Omega\bar{g}(\overline{\nabla}u,\overline{\nabla}u)\bar{g}(\overline{\nabla}V,\overline{\nabla}(u-\varphi))dx+\int_\Omega u\overline{\nabla}^2u(\overline{\nabla}V,\overline{\nabla}(u-\varphi))dx\\
&+\frac{K}{2}\int_\Omega u^2\bar{g}(\overline{\nabla}V,\overline{\nabla}u)dx+\frac{K}{2}\int_\Omega u^2\bar{g}(\overline{\nabla}V,\overline{\nabla}\varphi)dx,
\end{aligned}
\end{align}
where integrating by parts and \eqref{V1} gives
\begin{align*}
&-\frac{1}{2}\int_\Omega\bar{g}(\overline{\nabla}u,\overline{\nabla}u)\bar{g}(\overline{\nabla}V,\overline{\nabla}(u-\varphi))dx\\
=&-\frac{1}{2}\int_{\partial\Omega}\bar{g}(\overline{\nabla}u,\overline{\nabla}u)(u-\varphi)V_\nu dA+\int_\Omega (u-\varphi)\overline{\nabla}^2u(\overline{\nabla}V,\overline{\nabla}u)dx\\
&-\frac{n+1}{2}K\int_\Omega\bar{g}(\overline{\nabla}u,\overline{\nabla}u)V(u-\varphi)dx,
\end{align*}
by using \eqref{V1} and the Ricci identity, $\overline{\triangle}u+(n+1)Ku=1$ and $\overline{\text{Ric}}=nK\bar{g}$,
\begin{align*}
&\int_\Omega u\overline{\nabla}^2u(\overline{\nabla}V,\overline{\nabla}(u-\varphi))dx\\
=&\int_{\partial\Omega}u(u-\varphi)\overline{\nabla}^2u(\overline{\nabla}V,\nu)dA-\int_\Omega(u-\varphi)\overline{\nabla}^2u(\overline{\nabla}V,\overline{\nabla}u)dx+K\int_\Omega Vu(u-\varphi)\overline{\triangle}udx\\
&-\int_\Omega u(u-\varphi)\bar{g}(\overline{\nabla}V,\overline{\nabla}\,\overline{\triangle}u)dx-\int_\Omega u(u-\varphi)\overline{\text{Ric}}(\overline{\nabla}V,\overline{\nabla}u)dx\\
=&\int_{\partial\Omega}u(u-\varphi)\overline{\nabla}^2u(\overline{\nabla}V,\nu)dA-\int_\Omega(u-\varphi)\overline{\nabla}^2u(\overline{\nabla}V,\overline{\nabla}u)dx\\
&+K\int_\Omega Vu(u-\varphi)dx-(n+1)K^2\int_\Omega Vu^2(u-\varphi)dx+K\int_\Omega u(u-\varphi)\bar{g}(\overline{\nabla}V,\overline{\nabla}u)dx.
\end{align*}
Combining these computations, we obtain
\begin{align*}
\Lambda=&-\frac{1}{2}\int_{\partial\Omega}\bar{g}(\overline{\nabla}u,\overline{\nabla}u)(u-\varphi)V_\nu dA+\int_{\partial\Omega}u(u-\varphi)\overline{\nabla}^2u(\overline{\nabla}V,\nu)dA\\
&+K\int_\Omega Vu(u-\varphi)dx-(n+1)K^2\int_\Omega Vu^2(u-\varphi)dx+K\int_\Omega u(u-\varphi)\bar{g}(\overline{\nabla}V,\overline{\nabla}u)dx\\
&+\frac{K}{2}\int_\Omega u^2\bar{g}(\overline{\nabla}V,\overline{\nabla}u)dx+\frac{K}{2}\int_\Omega u^2\bar{g}(\overline{\nabla}V,\overline{\nabla}\varphi)dx-\frac{n+1}{2}K\int_\Omega\bar{g}(\overline{\nabla}u,\overline{\nabla}u)V(u-\varphi)dx.
\end{align*}
Substituting $\Lambda$ back into \eqref{beforeI}, we have
\begin{align}
\begin{aligned}
&\int_\Omega-Vu\,\overline{\triangle}Pdx\\
=&-\frac{1}{2}\int_{\partial\Omega}\bar{g}(\overline{\nabla}u,\overline{\nabla}u)(u-\varphi)V_\nu dA+\int_{\partial\Omega}u(u-\varphi)\overline{\nabla}^2u(\overline{\nabla}V,\nu)dA-\int_{\partial\Omega}VuP_\nu dA\\
&+\int_{\partial\Omega}VP(u_\nu-\varphi_\nu)dA+\int_{\partial\Omega} Vu\overline{\nabla}^2u(\overline{\nabla}\varphi,\nu)dA-\frac{1}{n+1}\int_{\partial\Omega} Vu\varphi_\nu dA\\
&+K\int_\Omega Vu(u-\varphi)dx-(n+1)K^2\int_\Omega Vu^2(u-\varphi)dx+K\int_\Omega u(u-\varphi)\bar{g}(\overline{\nabla}V,\overline{\nabla}u)dx\\
&+\frac{K}{2}\int_\Omega u^2\bar{g}(\overline{\nabla}V,\overline{\nabla}u)dx-\frac{K}{2}\int_\Omega u^2\bar{g}(\overline{\nabla}V,\overline{\nabla}\varphi)dx+K\int_{\partial\Omega} Vu^2\varphi_\nu dA\\
&+(n+1)K\int_\Omega VP(u-\varphi)dx-\frac{n+1}{2}K\int_\Omega\bar{g}(\overline{\nabla}u,\overline{\nabla}u)V(u-\varphi)dx.
\end{aligned}\label{afterI}
\end{align}
Again, expanding all $P$ and $P_\nu$ in \eqref{afterI}, we obtain
\begin{align*}
\begin{aligned}
&\int_\Omega-Vu\,\overline{\triangle}Pdx\\
=&-\frac{1}{2}\int_{\partial\Omega}\bar{g}(\overline{\nabla}u,\overline{\nabla}u)(u-\varphi)V_\nu dA+\int_{\partial\Omega}u(u-\varphi)\overline{\nabla}^2u(\overline{\nabla}V,\nu)dA\\
&+\frac{1}{2}\int_{\partial\Omega}\bar{g}(\overline{\nabla}u,\overline{\nabla}u)V(u-\varphi)_\nu dA-\int_{\partial\Omega} Vu\overline{\nabla}^2u(\overline{\nabla}(u-\varphi),\nu)dA-\frac{K}{2}\int_{\partial\Omega} Vu^2(u-\varphi)_\nu dA\\
&-\frac{n+1}{2}K^2\int_\Omega Vu^2(u-\varphi)dx+\frac{K}{2}\int_\Omega u^2\bar{g}(\overline{\nabla}V,\overline{\nabla}(u-\varphi))dx+K\int_\Omega u(u-\varphi)\bar{g}(\overline{\nabla}V,\overline{\nabla}u)dx,
\end{aligned}
\end{align*}
by virtue of integration by parts for the last term, we get
\begin{align*}
&K\int_\Omega u(u-\varphi)\bar{g}(\overline{\nabla}V,\overline{\nabla}u)dx\\
=&\frac{K}{2}\int_\Omega(u-\varphi)\bar{g}(\overline{\nabla}V,\overline{\nabla}u^2)dx\\
=&\frac{K}{2}\int_{\partial\Omega} u^2(u-\varphi)V_\nu dA-\frac{K}{2}\int_\Omega u^2(u-\varphi)\overline{\triangle}Vdx-\frac{K}{2}\int_\Omega u^2\bar{g}(\overline{\nabla}V,\overline{\nabla}(u-\varphi))dx\\
=&\frac{K}{2}\int_{\partial\Omega} u^2(u-\varphi)V_\nu dA+\frac{n+1}{2}K^2\int_\Omega Vu^2(u-\varphi)dx-\frac{K}{2}\int_\Omega u^2\bar{g}(\overline{\nabla}V,\overline{\nabla}(u-\varphi))dx,
\end{align*}
where we have used \eqref{V1} in the last equality. Therefore,
\begin{align}
\begin{aligned}
&\int_\Omega-Vu\,\overline{\triangle}Pdx\\
=&\frac{1}{2}\int_{\partial\Omega}\bar{g}(\overline{\nabla}u,\overline{\nabla}u)\left[V(u-\varphi)_\nu-(u-\varphi)V_\nu\right] dA+\int_{\partial\Omega}u\overline{\nabla}^2u\left((u-\varphi)\overline{\nabla}V-V\overline{\nabla}(u-\varphi),\nu\right)dA\\
&+\frac{K}{2}\int_{\partial\Omega} u^2\left[(u-\varphi)V_\nu-V(u-\varphi)_\nu\right] dA\\
=&\frac{1}{2}\int_{\partial\Omega}\bar{g}(\overline{\nabla}u,\overline{\nabla}u)\left[V(u-\varphi)_\nu-(u-\varphi)V_\nu\right] dA+\int_Tu\overline{\nabla}^2u\left((u-\varphi)\overline{\nabla}V-V\overline{\nabla}(u-\varphi),\bar{N}\right)dA\\
&+\frac{K}{2}\int_T u^2\left[(u-\varphi)V_{\bar{N}}-V(u-\varphi)_{\bar{N}}\right]dA,
\end{aligned}\label{tofinal}
\end{align}
where we note that $u=0$ on $\Sigma$.

Since $\partial_{\bar{N}}u=\kappa u+\widetilde{c},\,\,\partial_{\bar{N}}\varphi=\kappa \varphi+\widetilde{c}$ on $T$ and \eqref{V2}, we have
\begin{align}
\begin{aligned}
V(u-\varphi)_{\bar{N}}-(u-\varphi)V_{\bar{N}}=\kappa V(u-\varphi)-\kappa(u-\varphi)V=0,
\end{aligned} \label{Tzero}
\end{align}
and
\begin{align*}
\bar{g}\left((u-\varphi)\overline{\nabla}V-V\overline{\nabla}(u-\varphi),\bar{N}\right)=&(u-\varphi)V_{\bar{N}}-V(u-\varphi)_{\bar{N}}\\
=&\kappa(u-\varphi)V-\kappa V(u-\varphi)\\
=&0,
\end{align*}
which means that $(u-\varphi)\overline{\nabla}V-V\overline{\nabla}(u-\varphi)$ is a tangent vector field on $T$, hence by Lemma \ref{Hessianu}, we know
$$\overline{\nabla}^2u\left((u-\varphi)\overline{\nabla}V-V\overline{\nabla}(u-\varphi),\bar{N}\right)=0\,\,\text{on}\,\,T.$$
Therefore, \eqref{tofinal} becomes
\begin{equation*}
\int_\Omega-Vu\,\overline{\triangle}Pdx=\frac{1}{2}\int_\Sigma\bar{g}(\overline{\nabla}u,\overline{\nabla}u)\left[V(u-\varphi)_\nu-(u-\varphi)V_\nu\right] dA.
\end{equation*}
Note that \eqref{Tzero}, we have
\begin{align*}
\int_\Sigma\left[V(u-\varphi)_\nu-(u-\varphi)V_\nu\right] dA=&\int_{\partial\Omega}\left[V(u-\varphi)_\nu-(u-\varphi)V_\nu\right] dA\\
=&\int_\Omega\left[V\overline{\triangle}(u-\varphi)-(u-\varphi)\overline{\triangle}V\right]dx\\
=&0,
\end{align*}
where we have used $\overline{\triangle}u+(n+1)Ku=1,\,\,\overline{\triangle}\varphi+(n+1)K\varphi=1$ in $\Omega$ and \eqref{V1}.

So far, we have finished the proof of Proposition \ref{integralidentity}.
\end{proof}

\begin{proof}[Proof of Theorem \ref{Theorem}]
Note that in all cases, $V>0$ in $\Omega$ since $\Omega\subset B_{K,\kappa}^{\rm int}$. $V>0$ in $\Omega$, the subharmonicity of $P$ and the assumption $u\leq 0$ in $\Omega$ tell us that $$-Vu\,\overline{\triangle}P\geq 0 \,\,\text{in}\,\, \Omega.$$ Note that $\bar{g}(\overline{\nabla}u,\overline{\nabla}u)=c^2$ on $\Sigma$ since $\partial_\nu u=c$ on $\Sigma$, by taking $a=c^2$ in the integral identity \eqref{arbitrary}, we get $$\int_\Omega Vu\,\overline{\triangle}Pdx=0.$$
Hence $\overline{\triangle}P\equiv 0$ in $\Omega$ since $V>0$ and $u<0$ in $\Omega$. Since Lemma \ref{Laplician}, we know
\begin{equation}
\overline{\nabla}^2u=\left(\frac{1}{n+1}-Ku\right)\bar{g}\,\,\text{in}\,\,\Omega.\label{hess}
\end{equation}
By restricting \eqref{hess} on $\Sigma$ and using $u=0$ on $\Sigma$, we get $$c\,h_{ij}=\left(\overline{\nabla}^2u\right)_{ij}=\frac{1}{n+1}\bar{g}_{ij},$$
which means that $\Sigma$ must be part of an umbilical hypersurface with principal curvature $\frac{1}{(n+1)c}$, mean curvature $H=\frac{n}{(n+1)c}$.

Now we consider the angle characterization. Since $\Sigma$ and $S_{K,\kappa}$ are umbilical, the symmetry of umbilical hypersurfaces in conformal Euclidean space model tells us that the contact angle $\theta$ of $\Sigma$ with $S_{K,\kappa}$ must be a constant. Using $u_\nu=c$ on $\Sigma$, we have
\begin{align*}
c\int_\Sigma VdA&=\int_\Sigma Vu_\nu dA\\
&=\int_\Omega\left(V\overline{\triangle}u-u\overline{\triangle}V\right)dx+\int_T\left(uV_{\bar{N}}-Vu_{\bar{N}}\right)dA\\
&=\int_\Omega Vdx-\widetilde{c}\int_T VdA,
\end{align*}
where we have used $\overline{\triangle}u+(n+1)Ku=1$ in $\Omega$, $u_{\bar{N}}=\kappa u+\widetilde{c}$ on $T$ and \eqref{V1}, \eqref{V2}, hence
\begin{equation}\label{H2}
H=\frac{n}{(n+1)c}=\frac{n}{n+1}\frac{\int_\Sigma VdA}{\int_\Omega Vdx-\widetilde{c}\int_T VdA}.
\end{equation}
Combining \eqref{H2} and \eqref{meancurvature}, we get
$$\widetilde{c}=-\frac{n}{n+1}\cos\theta\frac{\int_TVdA}{\mathcal{M}\int_\Gamma\bar{g}(\nu,Y)ds},$$
using \eqref{div1} and \eqref{div2}, we may rearrange this to see that $\theta$ can be characterized by
\begin{align*}
\cos\theta&=-\frac{(n+1)\widetilde{c}}{n}\,\frac{\mathcal{M}\int_\Gamma\bar{g}(\nu,Y)ds}{\int_TVdA}\\
&=-\frac{(n+1)\widetilde{c}}{n}H\\
&=-\frac{(n+1)\widetilde{c}}{n}\frac{n}{(n+1)c}\\
&=-\frac{\widetilde{c}}{c}.
\end{align*}
\end{proof}

\hypertarget{AppendixA}{}\section*{Appendix A: Examples for BVP \eqref{Xie} with two horospheres as boundary.}
In \cite{guo2022partially}, the authors gave an example that the partially overdetermined BVP \eqref{Xie} admits a solution if the domain is bounded by two orthogonal horospheres ($\widetilde{c}=0$ and $\theta=\pi/2$). In this appendix, we would like to generalize that example. Precisely, we will illustrate that \eqref{Xie} admits a solution if the two boundary hypersurfaces are horospheres as well, while the contact angle $\theta$ can be arbitrary among $(0,\pi)$.

Consider two horospheres with principle curvature $\kappa=1$ in the Poincar\'{e} ball model:
\begin{align*}
  L_1:= & \left\{x\in\mathbb{B}^{n+1}~\Bigg|~|x'|^2+\left(x_{n+1}-\frac{1}{2}\right)^2=\frac{1}{4}\right\}, \\
  L_2:= & \left\{x\in\mathbb{B}^{n+1}~|~|x'|^2+(x_{n+1}+b)^2=(1-b)^2\right\},
\end{align*}
where $b\in (0,1/2)$. Let $\Omega$ be the domain bounded by $L_1$ and $L_2$, then $\partial\Omega=\overline{\Sigma}\cup T$, where $\overline{\Sigma}\subset L_1$ and $T\subset L_2$ (See Figure \ref{FigureA}). Applying the cosine law to the triangle $O_1O_2C$, we have
$$\cos\theta=-\cos(\pi-\theta)=-\frac{|O_1C|^2+|O_2C|^2-|O_1O_2|^2}{2|O_1C||O_2C|}=-\frac{1-3b}{1-b}.$$

\begin{center}
\renewcommand{\thefigure}{A}
\begin{figure}[h]
\vspace{-15pt}
\begin{tikzpicture}[scale=1.3]
    \draw[cyan,fill=gray!25] (0,-0.5) circle [radius=2];\draw[thick] (0,-0.5) circle [radius=0.01];\node[left] at (0,-0.5) {\scriptsize$O_2$};
    \draw[dashed,thick] (0,0) circle [radius=2.5];\draw[thick] (0,0) circle [radius=0.01];
    \node[left] at (0,-0.07) {\scriptsize$O$};\draw[dashed,->] (0,-2.5) to (0,2.6);\node[left] at (0,2.6) {\footnotesize{$x_{n+1}$}};
    \draw[red] (0,1.25) circle [radius=1.25];\draw[thick] (0,1.25) circle [radius=0.01];\node[left] at (0,1.25) {\scriptsize$O_1$};
    \draw[thick] (1.23718,1.07143) circle [radius=0.01];\node[below] at (1.23718,1.07143) {\scriptsize$C$};
    \node at (-1.36,1.32) {\scriptsize$L_1$};\node at (-2.15,-0.2) {\scriptsize$L_2$};
    \draw[dashed] (0,1.25)--(1.23718,1.07143)--(0,-0.5);\node at (0,0.5) {\scriptsize$\Omega$};
    \node at (0.7,1.5) {\scriptsize$T$};\node at (-0.7,0.08) {\scriptsize$\overline{\Sigma}$};
    \draw[->] (1.23718,1.07143)--(1.3,1.5068);\node[above] at (1.3,1.5068) {\scriptsize$\mu$};
    \draw[->] (1.23718,1.07143)--(1.6,0.7858);\node[right] at (1.6,0.7858) {\scriptsize$\bar{\nu}$};
    \node[right] at (1.23718,1.1) {\scriptsize$\theta$};\draw (1.28,1.0377) arc [radius=0.05,start angle=-60,end angle=108];
    \node[right] at (0,1.18) {\scriptsize$\frac{1}{2}$};\node[right] at (0,-0.5) {\scriptsize$-b$};
\end{tikzpicture}
\caption{$L_1$ and $L_2$ are two horospheres with contact angle $\theta$, and the shaded area is $B_{K,\kappa}^{\rm int}$.}
\label{FigureA}
\vspace{-15pt}
\end{figure}
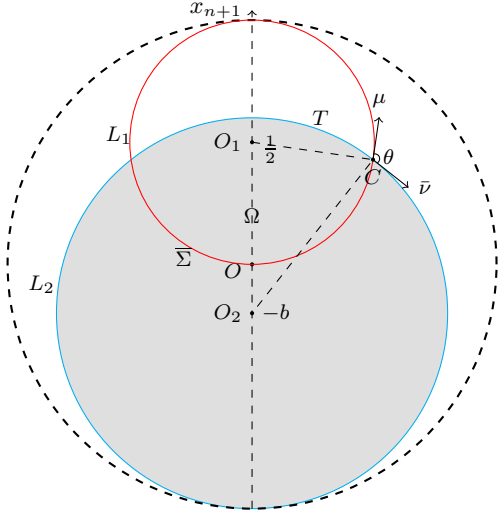
\end{center}

It has been shown in \cite{guo2022partially} that the function $u=\frac{1}{n+1}(V_3-V_1-1)$ satisfies
\begin{equation*}
\left\{
\begin{aligned}
\overline{\triangle}u-(n+1)u=1\quad&\text{in}\,\,\Omega,\\
u=0\quad&\text{on}\,\,\overline{\Sigma},\\
\partial_{\nu}u=\frac{1}{n+1}\quad&\text{on}\,\,\overline{\Sigma}.
\end{aligned}
\right.
\end{equation*}
where $V_3$ and $V_1$ are defined in \eqref{XVY3} and \eqref{XVY1} respectively. Note that the unit outward normal of $T$ is $$\bar{N}=\frac{1-|x|^2}{2(1-b)}(x+bE_{n+1}),$$
by direct computation,
\begin{align*}
  \partial_{\bar{N}}u-u &=\frac{1-|x|^2}{2(1-b)}\left[\sum_{i=1}^n x_i\frac{\partial u}{\partial x_i}+(x_{n+1}+b)\frac{\partial u}{\partial x_{n+1}}\right]-u \\
   &=\frac{3b|x|^2-x_{n+1}\left(|x|^2+2bx_{n+1}-1\right)-b}{(n+1)(1-b)(1-|x|^2)}\\
   &=\frac{1-3b}{(n+1)(1-b)},
\end{align*}
where we have used the equation of $T$ in the last equality.

Hence $u=\frac{1}{n+1}(V_3-V_1-1)$ is a solution of the partially overdetermined BVP \eqref{Xie}, with $c=\frac{1}{(n+1)}$ and $\widetilde{c}=\frac{1-3b}{(n+1)(1-b)}$. Moreover, the contact angle $\theta$ of $\overline{\Sigma}$ and $T$ satisfies
$$\cos\theta=-\frac{1-3b}{1-b}=-\frac{\widetilde{c}}{c},$$
which is consistent with the conclusion of Theorem \ref{Theorem}. We remark that $\theta$ can be arbitrary among $(0,\pi)$ since $b\in(0,1/2)$.

\hypertarget{AppendixB}{}\section*{Appendix B: Existence of weak solutions to the mixed BVP \eqref{BVP} for $\kappa=0$}
    We consider the Dirichlet-Neumann eigenvalue problem in a domain $\Omega$:
\begin{equation}\label{Dirichlet}\tag{B.1}
\left\{
\begin{aligned}
\overline{\triangle}u=-\lambda u\quad&\text{in}\,\,\Omega,\\
u=0\quad&\text{on}\,\,\overline{\Sigma},\\
\partial_{\bar{N}}u=0\quad&\text{on}\,\,T.
\end{aligned}
\right.
\end{equation}
The first Dirichlet-Neumann eigenvalue of \eqref{Dirichlet} can be variationally characterized by
$$\lambda_1(\Omega)=\text{inf}\left\{\frac{\int_\Omega\bar{g}(\overline{\nabla}u,\overline{\nabla}u)dx}{\int_\Omega u^2dx}\,\Bigg|\,0\neq u\in W_0^{1,2}(\Omega,\Sigma)\right\}.$$
If $S_{K,\kappa}$ is an umbilical hypersurface in $\mathbb{H}^{n+1}$ or $\mathbb{S}^{n+1}$ and $\Omega\subsetneqq B_{K,\kappa}^{\text{int}}$, then from \cite[Proposition 3.1-3.2]{guo2022partially} we know that
\begin{equation}\label{two}\tag{B.2}
\lambda_1(\Omega)>(n+1)K.
\end{equation}

    We now prove of the existence and uniqueness of weak solutions to the mixed BVP \eqref{BVP}.

\noindent\textbf{Proposition B.} Let $f\in C^\infty(\Omega), q\in C^\infty(T)$ and $\Omega\subsetneqq B_{K,\kappa}^{\text{int}}$. Then the mixed BVP
\begin{equation}\label{existence}\tag{B.3}
\left\{
\begin{aligned}
\overline{\triangle}u+(n+1)Ku=f\quad&\text{in}\,\,\Omega,\\
u=0\quad&\text{on}\,\,\overline{\Sigma},\\
\partial_{\bar{N}}u=q\quad&\text{on}\,\,T,
\end{aligned}
\right.
\end{equation}
admits a unique weak solution $u\in W_0^{1,2}(\Omega,\Sigma)$.
\begin{proof}
The weak solution to \eqref{existence} is defined to be $u\in W_0^{1,2}(\Omega,\Sigma)$ such that
\begin{equation*}
B[u,v]:=\int_\Omega\bar{g}(\overline{\nabla}u,\overline{\nabla}v)dx-(n+1)K\int_\Omega uvdx,\,\,\forall\,v\in W_0^{1,2}(\Omega,\Sigma),
\end{equation*}
and
\begin{equation*}
B[u,v]=-\int_\Omega fvdx+\int_T qvdA,\,\,\forall\,v\in W_0^{1,2}(\Omega,\Sigma).
\end{equation*}
If $K=-1$, then
\begin{equation*}
\begin{aligned}
B[u,u]&=\int_\Omega\bar{g}(\overline{\nabla}u,\overline{\nabla}u)dx+(n+1)\int_\Omega u^2dx\\
&\geq\int_\Omega\bar{g}(\overline{\nabla}u,\overline{\nabla}u)dx+\int_\Omega u^2dx\\
&=\|u\|^2_{W_0^{1,2}(\Omega,\Sigma)}.
\end{aligned}
\end{equation*}
If $K=1$, since $\lambda_1(\Omega)>n+1>1$, then
\begin{equation*}
\begin{aligned}
B[u,u]&=\int_\Omega\bar{g}(\overline{\nabla}u,\overline{\nabla}u)dx-(n+1)\int_\Omega u^2dx\\
&\geq\frac{\lambda_1(\Omega)-(n+1)}{\lambda_1(\Omega)+1}\int_\Omega\bar{g}(\overline{\nabla}u,\overline{\nabla}u)dx+\frac{n+2}{\lambda_1(\Omega)+1}\int_\Omega\bar{g}(\overline{\nabla}u,\overline{\nabla}u)dx-(n+1)\int_\Omega u^2dx\\
&\geq\frac{\lambda_1(\Omega)-(n+1)}{\lambda_1(\Omega)+1}\left(\int_\Omega\bar{g}(\overline{\nabla}u,\overline{\nabla}u)dx+\int_\Omega u^2dx\right)\\
&=\frac{\lambda_1(\Omega)-(n+1)}{\lambda_1(\Omega)+1}\|u\|^2_{W_0^{1,2}(\Omega,\Sigma)}.
\end{aligned}
\end{equation*}
Hence, $B[u,v]$ is coercive on $W_0^{1,2}(\Omega,\Sigma)$. The standard Lax-Milgram theorem holds for the weak formulation to \eqref{existence}. Therefore, \eqref{existence} admits a unique weak solution $u\in W_0^{1,2}(\Omega,\Sigma)$.
\end{proof}

\bibliographystyle{plain}
\bibliography{Reference}

\begin{thebibliography}{10}

\bibitem{aleksandrov1962uniqueness}
Aleksandr~D Aleksandrov.
\newblock Uniqueness theorems for surfaces in the large.
\newblock {\em Amer. Math. Soc. Transl.(2)}, 21:341--354, 1962.

\bibitem{beardon2012geometry}
Alan~F Beardon.
\newblock {\em The geometry of discrete groups}, volume~91.
\newblock Springer Science \& Business Media, 2012.

\bibitem{birindelli2013overdetermined}
I~Birindelli and F~Demengel.
\newblock Overdetermined problems for some fully non linear operators.
\newblock {\em Communications in Partial Differential Equations},
  38(4):608--628, 2013.

\bibitem{brandolini2008serrin}
Barbara Brandolini, Carlo Nitsch, Paolo Salani, and Cristina Trombetti.
\newblock Serrin-type overdetermined problems: an alternative proof.
\newblock {\em Archive for rational mechanics and analysis}, 190:267--280,
  2008.

\bibitem{buttazzo2011overdetermined}
Giuseppe Buttazzo and Bernd Kawohl.
\newblock Overdetermined boundary value problems for the $\infty$-{Laplacian}.
\newblock {\em International Mathematics Research Notices}, 2011(2):237--247,
  2011.

\bibitem{chen2022some}
Yimin Chen.
\newblock Some rigidity results on compact hypersurfaces with capillary
  boundary in hyperbolic space.
\newblock {\em arXiv preprint arXiv:2206.09062}, 2022.

\bibitem{cianchi2009overdetermined}
Andrea Cianchi and Paolo Salani.
\newblock Overdetermined anisotropic elliptic problems.
\newblock {\em Mathematische Annalen}, 345:859--881, 2009.

\bibitem{ciraolo2020serrin}
Giulio Ciraolo and Alberto Roncoroni.
\newblock Serrin's type overdetermined problems in convex cones.
\newblock {\em Calculus of Variations and Partial Differential Equations},
  59:1--21, 2020.

\bibitem{ciraolo2019serrin}
Giulio Ciraolo and Luigi Vezzoni.
\newblock On {S}errin's overdetermined problem in space forms.
\newblock {\em manuscripta mathematica}, 159:445--452, 2019.

\bibitem{farina2008remarks}
A~Farina and B~Kawohl.
\newblock Remarks on an overdetermined boundary value problem.
\newblock {\em Calculus of Variations and Partial Differential Equations},
  31:351--357, 2008.

\bibitem{farina2010flattening}
Alberto Farina and Enrico Valdinoci.
\newblock Flattening results for elliptic {PDE}s in unbounded domains with
  applications to overdetermined problems.
\newblock {\em Archive for rational mechanics and analysis}, 195:1025--1058,
  2010.

\bibitem{fragala2006overdetermined}
Ilaria Fragal{\`a}, Filippo Gazzola, and Bernd Kawohl.
\newblock Overdetermined problems with possibly degenerate ellipticity, a
  geometric approach.
\newblock {\em Mathematische Zeitschrift}, 254:117--132, 2006.

\bibitem{garofalo1989symmetry}
Nicola Garofalo and John~L Lewis.
\newblock A symmetry result related to some overdetermined boundary value
  problems.
\newblock {\em American Journal of Mathematics}, 111(1):9--33, 1989.

\bibitem{guo2022stable}
Jinyu Guo, Guofang Wang, and Chao Xia.
\newblock Stable capillary hypersurfaces supported on a horosphere in the
  hyperbolic space.
\newblock {\em Advances in Mathematics}, 409:108641, 2022.

\bibitem{guo2019partially}
Jinyu Guo and Chao Xia.
\newblock A partially overdetermined problem in a half ball.
\newblock {\em Calculus of Variations and Partial Differential Equations},
  58(5):160, 2019.

\bibitem{guo2022partially}
Jinyu Guo and Chao Xia.
\newblock A partially overdetermined problem in domains with partial umbilical
  boundary in space forms.
\newblock {\em Advances in Calculus of Variations}, (0), 2022.

\bibitem{jia2023characterization}
Xiaohan Jia, Zheng Lu, Chao Xia, and Xuwen Zhang.
\newblock A characterization of capillary spherical caps by a partially
  overdetermined problem in a half ball.
\newblock {\em arXiv preprint arXiv:2311.18581}, 2023.

\bibitem{jia2023rigidity}
Xiaohan Jia, Zheng Lu, Chao Xia, and Xuwen Zhang.
\newblock Rigidity and quantitative stability for partially overdetermined
  problems and capillary {CMC} hypersurfaces.
\newblock {\em arXiv e-prints}, pages arXiv--2311, 2023.

\bibitem{kumaresan1998serrin}
Somas Kumaresan and Jyotshana Prajapat.
\newblock Serrin's result for hyperbolic space and sphere.
\newblock 1998.

\bibitem{lieberman1986mixed}
Gary~M Lieberman.
\newblock Mixed boundary value problems for elliptic and parabolic differential
  equations of second order.
\newblock {\em Journal of Mathematical Analysis and Applications},
  113(2):422--440, 1986.

\bibitem{lu2012overdetermined}
Guozhen Lu and Jiuyi Zhu.
\newblock An overdetermined problem in {Riesz}-potential and fractional
  {L}aplacian.
\newblock {\em Nonlinear Analysis: Theory, Methods \& Applications},
  75(6):3036--3048, 2012.

\bibitem{molzon1991symmetry}
Robert Molzon.
\newblock Symmetry and overdetermined boundary value problems.
\newblock 1991.

\bibitem{pacella2020overdetermined}
Filomena Pacella, Giulio Tralli, et~al.
\newblock Overdetermined problems and constant mean curvature surfaces in
  cones.
\newblock {\em Rev. Mat. Iberoam}, 36(3):841--867, 2020.

\bibitem{serrin1971symmetry}
James Serrin.
\newblock A symmetry problem in potential theory.
\newblock {\em Archive for Rational Mechanics and Analysis}, 43:304--318, 1971.

\bibitem{wang2011characterization}
Guofang Wang and Chao Xia.
\newblock A characterization of the {Wulff} shape by an overdetermined
  anisotropic {PDE}.
\newblock {\em energy}, 22:13, 2011.

\bibitem{wang2019uniqueness}
Guofang Wang and Chao Xia.
\newblock Uniqueness of stable capillary hypersurfaces in a ball.
\newblock {\em Mathematische Annalen}, 374:1845--1882, 2019.

\bibitem{weinberger1971remark}
Hans~F Weinberger.
\newblock Remark on the preceding paper of {S}errin.
\newblock {\em Archive for Rational Mechanics and Analysis}, 43:319--320, 1971.

\end{thebibliography}

\end{document}